\documentclass[10pt,leqno]{article}
\usepackage{amsmath,amsthm,amsfonts,amssymb}

\usepackage[letterpaper]{geometry}

\usepackage{layout}

\newcommand{\Z}{{\mathbb Z}}
\newcommand{\E}{{\mathbb E}}

\newcommand{\eps}{{\varepsilon}}

\newcommand{\C}{{\mathbb C}}
\newcommand{\R}{{\mathbb R}}

\newcommand\cA{{\cal  A}}
\newcommand\cB{{\cal  B}}

\newcommand\cU{{\cal  U}}
\newcommand\cH{{\cal  H}}
\newcommand\cV{{\cal  V}}

\newcommand\cC{{\cal  C}}
\newcommand\cI{{\cal  I}}

\newcommand\cG{{\cal  G}}

\newcommand\cL{{\cal  L}}
\newcommand\cN{{\cal  N}}
\newcommand\cE{{\cal  E}}
\newcommand\cF{{\cal  F}}
\newcommand\cP{{\cal  P}}

\newcommand\cO{{\cal O}}

\newcommand\cM{{\mathcal M}}

\newcommand{\bR}{\mathbb{R}}

\newcommand{\bL}{\mathbb{L}}
\newcommand{\bP}{\mathbb{P}}
\newcommand{\bE}{\mathbb{E}}

\newcommand{\bQ}{\mathbb{Q}}

\newcommand{\bA}{\mathbb{A}}
\newcommand{\bM}{\mathbb{M}}

\newcommand{\bT}{\mathbb{T}}
\newcommand{\bH}{\mathbb{H}}

\newcommand{\bC}{\mathbb{C}}
\newcommand{\bZ}{\mathbb{Z}}

\newcommand{\tA}{\tilde A}

\newcommand{\up}{{\underline p}}
\newcommand{\uq}{{\underline q}}
\newcommand{\ur}{{\underline r}}

\newcommand{\uv}{{\underline v}}

\newcommand{\utau}{\underline \tau}

\newcommand{\ueta}{\underline \eta}

\newcommand{\uG}{\underline G}

\newcommand{\uomega}{{\underline \omega}}

\def\eps{\epsilon }

\def\D{\partial }
\newcommand\adots{\mathinner{\mkern2mu\raise1pt\hbox{.}
\mkern3mu\raise4pt\hbox{.}\mkern1mu\raise7pt\hbox{.}}}

\newcommand\br{\begin{remark}}
\newcommand\er{\end{remark}}
\newcommand\bp{\begin{pmatrix}}
\newcommand\ep{\end{pmatrix}}
\newcommand\be{\begin{equation}}
\newcommand\ee{\end{equation}}
\newcommand\ba{\begin{equation}\begin{aligned}}
\newcommand\ea{\end{aligned}\end{equation}}

\newcommand{\bap}{\begin{app}}
\newcommand{\eap}{\end{app}}
\newcommand{\begs}{\begin{exams}}
\newcommand{\eegs}{\end{exams}}
\newcommand{\beg}{\begin{example}}
\newcommand{\eeg}{\end{exaplem}}
\newcommand{\bpr}{\begin{proposition}}
\newcommand{\epr}{\end{proposition}}
\newcommand{\bt}{\begin{theorem}}
\newcommand{\et}{\end{theorem}}
\newcommand{\bc}{\begin{corollary}}
\newcommand{\ec}{\end{corollary}}
\newcommand{\bl}{\begin{lem}}
\newcommand{\el}{\end{lem}}
\newcommand{\bd}{\begin{definition}}
\newcommand{\ed}{\end{definition}}
\newcommand{\brs}{\begin{remarks}}
\newcommand{\ers}{\end{remarks}}

\newtheorem{theo}{Theorem}[section]
\newtheorem{prop}[theo]{Proposition}

\newtheorem{lem}[theo]{Lemma}
\newtheorem{defn}[theo]{Definition}
\newtheorem{ass}[theo]{Assumption}

\newtheorem{exam}[theo]{Example}
\newtheorem{rem}[theo]{Remark}

\newtheorem{exams}[theo]{Examples}

\newtheorem{definition}[theo]{Definition}

\numberwithin{equation}{section}

 \title{Resonant leading term geometric optics expansions with boundary layers for quasilinear hyperbolic boundary problems}

\author{\sc \small
Matthew Hernandez\thanks{
University of North Carolina-Chapel Hill;
mbh3@princeton.edu}}

\begin{document}
\maketitle

\begin{abstract}
We construct and justify leading order weakly nonlinear geometric optics expansions for nonlinear hyperbolic initial value problems, including the compressible Euler equations. The technique of simultaneous Picard iteration is employed to show approximate solutions tend to the exact solutions in the small wavelength limit. Recent work \cite{cgw} by Coulombel, Gues, and Williams studied the case of reflecting wave trains whose expansions involve only real phases. We treat generic boundary frequencies by incorporating into our expansions both real and nonreal phases. Nonreal phases introduce difficulties such as approximately solving complex transport equations and result in the addition of boundary layers with exponential decay. This also prevents us from doing an error analysis based on almost-periodic profiles as in \cite{cgw}.
\end{abstract}

\tableofcontents

\section{Introduction}
\emph{\quad}In this paper we consider quasilinear hyperbolic fixed boundary problems on the domain $\overline {\mathbb{R}}^{d+1}_+ = \{x=(x',x_d)=(t,y,x_d)=(t,x''):x_d\geq 0\}$ for a class of equations which includes, in particular, the compressible Euler equations. The class consists of systems of the following form:
\begin{align} \label{new51}
\begin{split}
& \sum^d_{j=0}A_j(v_\epsilon)\partial_{x_j}v_\epsilon=f(v_\epsilon),\\
&b(v_\epsilon)|_{x_d=0}=g_0+\epsilon G\left(x',\frac{x'\cdot\beta}{\epsilon}\right),\\
&v_\epsilon = u_0 \text{ in } t < 0.
\end{split}
\end{align}
Here, $A_0=I$, and we assume $A_j\in C^\infty(\mathbb{R}^N,\mathbb{R}^{N^2})$, $f\in C^\infty(\mathbb{R}^N,\mathbb{R}^N)$, and $b\in C^\infty(\mathbb{R}^N,\mathbb{R}^p)$. For the boundary data, we take $G(x',\theta_0)\in C^\infty(\mathbb{R}^d\times\mathbb{T}^1,\mathbb{R}^p)$ periodic in $\theta_0$ and supported in $\{x_0\geq 0\}$ and the boundary frequency $\beta=(\beta_0,\ldots,\beta_{d-1})\in\mathbb{R}^d\setminus 0$.

We construct leading order weakly nonlinear geometric optics expansions and justify the expansions by establishing their convergence in $L^\infty$ to the unique exact solutions in the large frequency limit, i.e. the small wavelength limit, represented by $\eps\rightarrow 0$. This paper completes the treatment of generic $\beta$ in the following sense: we allow for all $\beta$ but that in the glancing region, a set of measure zero which is often regarded as a singular case. The case of $\beta$ belonging to the hyperbolic region, the interior of a cone determined by the $A_j$, was treated in \cite{cgw}, where the authors constructed leading order expansions of highly oscillatory wavetrains with real phase functions. Our construction includes similar wavetrains with real phases, but in handling more general $\beta$ we are required to add to the expansions exponentially decaying oscillations, which we refer to as \emph{elliptic boundary layers}, featuring nonreal phase functions.

\subsection{The problem and its reformulations}

\emph{\quad}First we motivate and outline the relevant reformulations of the problem \eqref{new51} and the form of the expansion.

\subsubsection{The solution as a perturbation}

\emph{\quad} Seeking $v_\eps = u_0+\eps u_\eps$, a perturbation of a constant state $u_0$ with $f(u_0)=0$ and $b(u_0)=g_0$, we translate \eqref{new51} to the following problem for $u_\eps$ (where the $A_j$ have been altered accordingly:)
\begin{align}\label{new53}
\begin{split}
&(a)\;P(\eps u_\epsilon,\partial_x) u_\epsilon:=\sum^d_{j=0}A_j(\epsilon u_\epsilon)\partial_{x_j}u_\epsilon=\mathcal{F}(\epsilon u_\epsilon)u_\epsilon,\\ 
&(b)\;B(\epsilon u_\epsilon)u_\epsilon|_{x_d=0}=G\left(x',\frac{x'\cdot \beta}{\epsilon}\right),\\
&(c)\;u_\epsilon = 0 \text{ in } t<0.
\end{split}
\end{align}
where $B(v)$ and $\cF(v)$ denote the $p\times N$ and $N\times N$ real matrices, $C^\infty$ in $v$, defined by
\begin{align}\label{new55}
B(v)v=b(u_0+v)-b(u_0),\quad \cF(v)v=f(u_0+v).
\end{align}
Let us define the operator
\begin{align}\label{a8a}
L(\partial_x):=P(0,\partial_x)=\partial_t+\sum_{j=1}^dA_j(0)\partial_{x_j}.
\end{align}

We assume that $L(\partial_x)$ is hyperbolic with characteristics of constant multiplicity, making the following:
\begin{ass}
\label{assumption1}
The matrix $A_0=I$.  For an open neighborhood $\cal{O}$ of  $0\in\R^N$, there exists an integer $q \ge 1$, some real functions $\lambda_1,\dots,\lambda_q$ that are $C^\infty$ on $\cal{O} \;\times$ $\R^d\setminus\{0\}$, homogeneous of degree $1$, and analytic in $\xi$, and there exist some positive integers $\nu_1,\dots,\nu_q$
such that:
\begin{equation}\label{as1}
\det \Big[ \tau \, I+\sum_{j=1}^d \xi_j \, A_j(u) \Big] =\prod_{k=1}^q \big( \tau+\lambda_k(u,\xi) \big)^{\nu_k}
\end{equation}
for $u\in\cal{O}$, $\xi=(\xi_1,\dots,\xi_d)\in\R^d\setminus\{0\}$.
Moreover the eigenvalues $\lambda_1(u,\xi),\dots,\lambda_q(u,\xi)$ are semi-simple (their algebraic multiplicity
equals their geometric multiplicity) and satisfy $\lambda_1(u,\xi)<\dots<\lambda_q(u,\xi)$ for all $u\in\cal{O}$, $\xi \in \R^d
\setminus \{ 0\}$.
\end{ass}

Also, for our study we consider a noncharacteristic boundary $\{x_d=0\}$:
\begin{ass}
\label{assumption2}
For $u\in \cal{O}$ the matrix $A_d(u)$ is invertible and the matrix $B(u)$ has maximal rank, its rank $p$ being equal to the number
of positive eigenvalues of $A_d(u)$ (counted with their multiplicity).
\end{ass}

We now provide some notation regarding frequencies $\tau-i\gamma\in\C$ and $\eta\in\R^{d-1}$ dual to the variables $t$ and $y$, respectively. We define the matrix

\begin{align}
\cA(\zeta):= - i\, A_d^{-1}(0) \left( \tau \, I +\sum_{j=1}^{d-1} \eta_j \, A_j(0) \right),\quad \zeta:=(\tau-i\gamma,\eta)\in\C\times\R^{d-1},
\end{align}
and the sets of frequencies
\begin{align*}
& \Xi := \Big\{ (\tau-i\gamma,\eta) \in \C \times \R^{d-1} \setminus (0,0) : \gamma \ge 0 \Big\} \, ,
& \Sigma := \Big\{ \zeta \in \Xi : \tau^2 +\gamma^2 +|\eta|^2 =1 \Big\} \, ,\\
& \Xi_0 := \Big\{ (\tau,\eta) \in \R \times \R^{d-1} \setminus (0,0) \Big\} = \Xi \cap \{ \gamma = 0 \} \, ,
& \Sigma_0 := \Sigma \cap \Xi_0 \, .
\end{align*}
With these in mind, we define the symbol
\begin{align}
L(\tau,\xi):=\tau I+\sum^d_{j=1} \xi_j A_j(0).
\end{align}

The following is a result due to Kreiss \cite{kreiss} in the case of strict hyperbolicity, i.e. when all the eigenvalues in Assumption \ref{assumption1} have multiplicity $\nu_j=1$, and to M\'{e}tivier \cite{metivier} in our more general case.
\begin{prop}[\cite{kreiss,metivier}]
\label{thm1}
Let Assumptions \ref{assumption1} and \ref{assumption2} be satisfied. Then for all $\zeta \in \Xi \setminus
\Xi_0$, the matrix ${\mathcal A}(\zeta)$ has no purely imaginary eigenvalue and its stable subspace $\E^s (\zeta)$ has dimension $p$. Furthermore, $\E^s$ defines an analytic vector bundle over $\Xi \setminus \Xi_0$ that can
be extended as a continuous vector bundle over $\Xi$.
\end{prop}
For $(\tau,\eta)\in\Xi_0$, we define $\E^s(\tau,\eta)$ to be the continuous extension obtained in Proposition \ref{thm1} of $\E^s$ to $(\tau,\eta)$.

Now we describe our assumption of \emph{uniform stability}, defined as in \cite{kreiss,cp}:
\begin{defn}\label{unifStability}
The problem \eqref{new53} is \emph{uniformly stable} at $u=0$ if the linearized operators $(L(\partial_x),B(0))$ at
$u=0$ are such that
\begin{align}
B(0):\E^s(\tau-i\gamma,\eta)\to\C^p \text{ is an isomorphism for all }(\tau-i\gamma,\eta)\in\Sigma.
\end{align}
\end{defn}

\begin{ass}\label{assumption3}
The problem \eqref{new53} is \emph{uniformly stable at $u=0$}.
\end{ass}

We now discuss an important example, the Euler equations, which satisfies the above assumptions so that our main results may be applied. We remark, however, that we also require Assumption \ref{assumption4}, an assumption on the boundary frequency $\beta$, which is explained in Section \ref{approxSolnConstruc}.

\begin{exam}[Euler equations] \label{EulerEqns}
The following are the isentropic, compressible Euler equations in three space dimensions on the half space $\{x_3\geq 0\}$, in the unknowns density $\rho$ and velocity $u=(u_1,u_2,u_3)$:
\begin{align}\label{a7aa}
\partial_t\begin{pmatrix}\rho\\\rho u_1\\\rho u_2\\\rho u_3\end{pmatrix}+\partial_{x_1}\begin{pmatrix}\rho u_1\\\rho u_1^2+p(\rho)\\\rho u_2u_1\\ \rho u_3 u_1\end{pmatrix}+\partial_{x_2}\begin{pmatrix}\rho u_2\\\rho u_1 u_2\\\rho u_2^2+p(\rho)\\\rho u_3u_2\end{pmatrix}+\partial_{x_3}\begin{pmatrix}\rho u_3\\\rho u_1u_3\\\rho u_2u_3\\\rho u_3^2+p(\rho)\end{pmatrix}=\begin{pmatrix}0\\0\\0\\0\end{pmatrix},
\end{align}
where $p(\rho)$ is the pressure. The hyperbolicity assumption, Assumption \ref{assumption1}, is satisfied in the region of state space where $\rho>0$, $c^2=p'(\rho)>0$. In this case the eigenvalues $\lambda_k(\rho,u,\xi)$ are
\begin{align}\label{EulerEvals}
\lambda_1=u\cdot\xi-c|\xi|,\; \lambda_2= u\cdot \xi, \;\lambda_3=u\cdot\xi+c|\xi|, \text{ with }(\nu_1,\nu_2,\nu_3)=(1,2,1).
\end{align}

For this problem, the boundary condition we impose is the natural ``residual boundary condition,'' which is obtained in the vanishing viscosity limit of the compressible Navier-Stokes equations with Dirichlet boundary conditions. Fixing any constant state $(\rho,u)$ with $u_3\notin\{0,-c,  c\}$ about which to linearize the problem, so that $(\rho,u)=u_0$ in \eqref{new51}, we have noncharacteristic boundary $\{x_3=0\}$ for the system. Consider in particular Assumption \ref{assumption3} in each of the following cases:

(a)\textbf{Subsonic outflow: $u_3<0$, $|u_3|<c$}. In this case there is exactly one positive eigenvalue of $A_3(\rho,u)$ ($p=1$), so we need one scalar boundary condition. Taking $b(\rho,u)=u_3$ in \eqref{new51}, we have $B(0)=\begin{bmatrix} 0&0&0&1 \end{bmatrix}$ and Assumption \ref{assumption3} is satisfied.

(b)\textbf{Subsonic inflow: $0<u_3<c$}. For this we get $p=3$ and boundary condition $b(\rho,u)=(\rho u_3,u_1,u_2)$ and linearized operator
\begin{align}
B(0)(\dot\rho, \dot u)=(\dot \rho u_3+\rho \dot u_3, \dot u_1, \dot u_2),
\end{align}
which satisfies Assumption \ref{assumption3}.

(c)\textbf{Supersonic inflow: $0<c<u_3$}. This case is trivial, with $p=4$ and $B(0)$ the $4\times 4$ identity matrix, and so Assumption \ref{assumption3} holds.

(d)\textbf{Supersonic outflow: $u_3<0$, $|u_3|>c$}. This is another trivial case, with $p=0$, where $B(0)$ is absent, meaning Assumption \ref{assumption3} holds vacuously.

With clear modifications of the above statements, the same holds for the 2D Euler equations, which are in fact strictly hyperbolic. For complete proofs and discussion verifying these cases, we refer the reader to \cite{gmwz}, Section 5.
\end{exam}

To study geometric optics for nonlinear problems, it is important to establish the existence of exact solutions on a fixed time interval independent of the wavelength $\epsilon$. For the system \eqref{new53} and thus \eqref{new51}, this was proved by M. Williams in \cite{williams2} through the use of the singular system, which we discuss further in section \ref{testlabel}.

\subsubsection{The approximate solution constructed in terms of an $M$-periodic function}\label{approxSolnConstruc}

\emph{\quad} The goal of this paper is to obtain qualitative information about the exact solution to \eqref{new53} by explicitly constructing an approximate solution with a geometric optics expansion that exhibits the qualitative information, and then showing that this approximate solution tends to the exact solution as $\eps\rightarrow 0$.

\textbf{Regular boundary frequencies.}

We now discuss an assumption regarding the boundary frequency $\beta=(\utau,\ueta)\in\bR\times\bR^{d-1}\setminus(0,0)$ which is relevant to the construction of the approximate solution. The planar real phase $\phi_0$ is defined on the boundary by
\begin{align}
\phi_0(x'):=\utau t + \ueta\cdot y = \beta\cdot x'.
\end{align}
Note the relationship to the boundary data \eqref{new53}(b). The oscillations on the boundary associated to $\phi_0$ result in oscillations associated to some planar phases $\phi_m$, which are characteristic for $L(\partial_x)$ and which have trace on the boundary equal to $\phi_0$. Such oscillations largely motivate the ansatz for our approximate solution, dependent on the planar phases $\phi_m$. The following assumption allows us to obtain the $\phi_m$.

\begin{ass}\label{assumption4}
  For $\zeta=(\tau-i\gamma,\eta)\in\C\times\R^{d-1}$  consider the matrix
  \begin{equation}
\label{new32}
\frac{1}{i}{\mathcal A}(\tau-i\gamma,\eta)= - \, A_d^{-1}(0) \left( (\tau-i\gamma) \, I +\sum_{j=1}^{d-1} \eta_j \, A_j(0) \right).
\end{equation}
Let $\uomega_m$, $m=1,\dots,M$ denote the distinct eigenvalues of $\frac{1}{i}{\mathcal A}(\utau,\ueta)$.  We suppose that each
$\uomega_m$ is a semisimple eigenvalue with multiplicity denoted by $\mu_m$.  Moreover, we assume there is a conic neighborhood $\cO$ of $\beta$ in $\C\times \R^{d-1}\setminus \{0\}$ on which the eigenvalues of $\frac{1}{i}{\mathcal A}(\zeta)$ are semisimple and given by smooth functions $\omega_m(\zeta)$, $m=1,\dots, M$ where $\uomega_m=\omega_m(\beta)$ and $\omega_m(\zeta)$ is of constant multiplicity $\mu_m$.
\end{ass}
\begin{rem}
\emph{Using Assumption \ref{assumption4}, one can show that the smooth functions $\omega_m(\zeta)$ are in fact analytic on a conic neighborhood of $\beta$. In particular, the analyticity is used to prove Lemma \ref{SubspaceDecompLem}.}
\end{rem}
We call $\beta$ a \emph{regular boundary frequency} provided Assumption \ref{assumption4} holds. Now, we define the phase functions
\begin{align}\label{new64}
\phi_m(x) := \phi_0(x')+\uomega_m x_d = (\beta,\uomega_m)\cdot x  ,
\end{align}
and we have the complex characteristic vector field associated to $\phi_m$:
\begin{align}\label{charvectorfield}
X_{\phi_m}:=\partial_{x_d}+\sum^{d-1}_{j=0}-\partial_{\xi_j}\omega_m(\beta)\partial_{x_j},
\end{align}
which is real for real $\uomega_m$. Moreover, for each $\uomega_m$ which is real, Assumption \ref{assumption1} implies there is a unique $k_m\in\{1,\dots,q\}$ such that
$\utau+\lambda_{k_m}(\ueta,\uomega_m)=0$, and in this case we have $\mu_m=\nu_{k_m}$.\footnote{Assumption \ref{assumption1} has no immediate implication for the multiplicity of nonreal $\uomega_m$.} With this in mind, we make the following definition:

\begin{definition} (i) For $m$ such that $\uomega_m$ is real, we call $(\beta,\uomega_m)$ a \emph{hyperbolic mode} if 
\begin{align}\label{c}
\partial_{\xi_d}\lambda_{k_m}(\ueta,\uomega_m)\neq 0.
\end{align}

(ii) For $m$ with nonreal $\uomega_m$, we call $(\beta,\uomega_m)$ an \emph{elliptic mode}.
\end{definition}
\begin{rem}\label{modeProps}
(i) \textup{For each real $\uomega_m$, Assumption \ref{assumption4} guarantees that $(\beta,\uomega_m)$ is a hyperbolic mode. That \eqref{c} holds is a consequence of the semisimplicity of $\uomega_m$.\footnote{If $\uomega_m$ is real and  $\partial_{\xi_d}\lambda_{k_m}(\ueta,\uomega_m)= 0$, we refer to $(\beta,\uomega_m)$ as a \emph{glancing} mode. An explanation of how \eqref{c} follows from semisimplicity can be found in the proof of Lemma 2.7 of \cite{metivier}.} Also, the condition \eqref{c} and the implicit function theorem imply that $\omega_m$ is real and of multiplicity $\nu_{k_m}$  in a neighborhood of $\beta$. Thus, there is a slight redundancy in Assumption \ref{assumption4} when $\uomega_m$ is real.}

(ii) \textup{When $\uomega_m$ is nonreal, the fact that the $A_j$ are real implies that $\frac{1}{i}{\mathcal A}(\beta)$ also has the complex conjugate $\overline{\uomega}_m$ as an eigenvalue, and thus $\overline{\uomega}_m=\uomega_{m'}$ for some $m'\neq m$. Furthermore, if the vector $r\in\bC^N$ is an eigenvector of $\frac{1}{i}{\mathcal A}(\beta)$ associated to $\uomega_m$, then $\overline{r}$ is an eigenvector associated to $\overline{\uomega}_m=\uomega_{m'}$.}

(iii) \textup{The boundary frequency $\beta$ lies in the hyperbolic region (as defined in \cite{cgw}) if and only if Assumption \ref{assumption4} holds with all $\uomega_m$ real. The elliptic region consists of all $\beta$ such that Assumption \ref{assumption4} holds with all $\uomega_m$ nonreal.}
\end{rem}

\begin{exam}
We return to the 3D Euler equations, linearizing about $(\rho,u)=(\rho,u_1,u_2,u_3)$, as in Example \ref{EulerEqns}.

For $|u_3|>c$ the set of regular boundary frequencies is all of $\R^3\setminus 0$. In the case $|u_3|<c$, one finds that the set of regular boundary frequencies is 
\begin{align}
\left\{(\tau,\eta)\in\bR^3: |\tau+u_1\eta_1+u_2\eta_2|\neq \sqrt{c^2-u_3^2}\;|\eta|\right\}.
\end{align}

We remark that in the former case, the hyperbolic region is also $\cH=\R^3\setminus 0$, and in the latter case $\cH=\left\{(\tau,\eta)\in\bR^3: |\tau+u_1\eta_1+u_2\eta_2|> \sqrt{c^2-u_3^2}\;|\eta|\right\}$.
\end{exam}

For real $\uomega_m$ we have the associated real group velocity:
\begin{align}\label{groupvelocity}
{\bf v}_m := \nabla \lambda_{k_m} ( \ueta, \uomega_m),
\end{align}
which has the following relationship with the characteristic vector field \eqref{charvectorfield}:
\begin{align}\label{b24aa}
\partial_{\xi_0}\omega_m(\beta)=-\frac{1}{\partial_{\xi_d}\lambda_{k_m}(\ueta,\uomega_m)},\;\partial_{\xi_j}\omega_m(\beta)=-\frac{\partial_{\xi_j}\lambda_{k_m}(\ueta,\uomega_m)}{\partial_{\xi_d}\lambda_{k_m}(\ueta,\uomega_m)},\; j=1,\dots,d-1.
\end{align}
Observe that each group velocity ${\bf v}_m$ can be thought of as either incoming or outgoing with respect to the interior of the domain $\R^d_+$, in particular since the last coordinate of ${\bf v}_m$ is nonzero, by \eqref{c}. With this in mind, we classify the phases in the following way:
\begin{definition}
\label{def2}
For real $\uomega_m$, the phase $\phi_m$ is incoming if the group velocity ${\bf v}_m$ is incoming (that is, $\partial_{\xi_d} \lambda_{k_m}
(\underline{\eta},\underline{\omega}_m)>0$), and it is outgoing if the group velocity ${\bf v}_m$ is outgoing
($\partial_{\xi_d} \lambda_{k_m} (\underline{\eta},\underline{\omega}_m) <0$).
\end{definition}
With this classification of the real phases, we let $\cI$ denote the set of indices $m\in\{1,\ldots,M\}$ such that $\phi_m$ is incoming and $\cO$ the set of $m$ such that $\phi_m$ is outgoing. We classify the remaining complex phases $\phi_m$, which correspond to nonreal $\uomega_m$, by the distinction that $\cP$ is the set of $m$ such that $\textrm{Im }\uomega_m>0$ and $\cN$ is the set of $m$ such that $\textrm{Im }\uomega_m<0$. We thus form the partition
\begin{align}
\{1,\ldots,M\}=\cI\cup\cO\cup\cP\cup\cN.
\end{align}

Recall the $\omega_m(\tau,\eta)$ are eigenvalues of $\frac{1}{i}\cA(\tau,\eta)$, meaning the $i\omega_m(\tau,\eta)$ are the eigenvalues of $\cA(\tau,\eta)$. Consider also the fact that $\textrm{Im }\omega_m(\tau,\eta)>0$ if and only if $\textrm{Re } i\omega_m(\tau,\eta)<0$. From these observations, we see that the stable subspace $\E^s(\tau,\eta)$ of $\cA(\tau,\eta)$ must contain each of the eigenspaces corresponding to some $\omega_m(\tau,\eta)$ with $\textrm{Re } i\uomega_m<0$, so that $\E^s(\utau,\ueta)$ contains the subspaces $\textrm{Ker }L(d\phi_m)$ for $m\in\cP$. The incoming phases, corresponding to $m\in\cI$, also play a role in setting up a decomposition for the stable subspace $\E^s$ at the boundary, which is established in the following lemma.

\begin{lem}\label{SubspaceDecompLem}
The stable subspace $\E^s(\utau,\ueta)$ admits the decomposition:
\textup{
\begin{align}\label{sSubspaceDecomp}
\E^s(\utau,\ueta)=\oplus_{m\in\cI\cup\cP}\textrm{Ker }L(d\phi_m),
\end{align}
}
where, in the decomposition, the vector spaces \textup{$\textrm{Ker }L(d\phi_m)$} for $m\in\cI$ admit a basis of real vectors.
\end{lem}
\begin{proof}
It is easy to show that the subspaces $\textrm{Ker }L(d\phi_m)$ for $m\in\cP$ are in $\E^s(\utau,\ueta)$. We will show that this is also the case for the subspaces $\textrm{Ker }L(d\phi_m)$ with $m\in\cI$, from which the result will follow. Since $\E^s(\utau,\ueta)$ is close to $\E^s(\utau-i\gamma,\ueta)$ for small $\gamma>0$, in accordance with Proposition \ref{thm1}, it will suffice to show for $m\in\cI$ that $i\omega_m(\utau-i\gamma,\ueta)$, close to $i\uomega_m$, satisfies
\begin{align}\label{stabilityIncoming}
\textrm{Re }i\omega_m(\utau-i\gamma,\ueta)<0.
\end{align}
Recalling \eqref{b24aa}, we see that since $\phi_m$ is incoming, i.e. $m\in \cI$, we have
\begin{align}\label{evalderiv}
\partial_{\zeta_0} \omega_m (\utau,\ueta)<0.
\end{align}
Using analyticity of $\omega_m(\zeta)$ with \eqref{evalderiv}, since $\omega_m(\utau,\ueta)$ has imaginary part equal to zero, it follows that $\omega_m(\utau-i\gamma,\ueta)$ has positive imaginary part for small positive $\gamma$. Therefore \eqref{stabilityIncoming} holds. The statement that each of the vector spaces $\textrm{Ker }L(d\phi_m)$ for $m\in\cI$ admits a basis of real vectors follows from Lemma \ref{lem2}, which was proved in \cite{cg}.
\end{proof}

\textbf{The ansatz. }

The following lemma, proved in \cite{cg}, provides a decomposition which is key to the construction of our ansatz and will serve us later in the construction of the projectors which give us the profile equations.
\begin{lem}
\label{lem2}
The space $\bC^N$ admits the decomposition:
\begin{equation}
\label{d}
\C^N = \oplus_{m=1}^M \, \text{\rm Ker } L({ d} \phi_m)
\end{equation}
and each vector space in \eqref{d} with $m\in\cI\cup\cO$ admits a basis of real vectors. If we let $P_1,\dots,P_M$ denote the projectors associated with the
decomposition \eqref{d}, then for all $m=1,\dots,M$, there holds $\text{\rm Im } A_d^{-1}L
({ d} \phi_m) = \text{\rm Ker } P_m$.
\end{lem}
For each $m=1,\ldots,M$, we choose a basis of $\textrm{Ker }L(d\phi_m)$:
\begin{align}
r_{m,k},\,k=1,\ldots,\mu_m,
\end{align}
where for real $\uomega_m$, we take a basis of real $r_{m,k}$. Now we are in a position to give the ansatz for our approximate solution to \eqref{new53}.
\begin{align}\label{new63}
u^a_\eps(x)=\uv(x)+\sum^M_{m=1}\sum^{\mu_m}_{k=1}\sigma_{m,k}\left(x,\frac{\phi_m(x)}{\eps}\right)r_{m,k},
\end{align}
where $\uv(x)$ and the $\sigma_{m,k}(x,\theta_m)$ are $C^1$ functions. Additionally, we require that each of the $\sigma_{m,k}(x,\theta_m)$ is periodic in $\theta_m$ with mean $0$, and we will refer to these as the \emph{periodic profiles}.
With these, we plug \eqref{new63} into \eqref{new53}(a) and expand the error in ascending powers of $\eps$:
\begin{align}\label{new66}
P(\eps u^a_\eps,\partial_x)u^a_\eps-\cF(\eps u^a_\eps)u^a_\eps = \eps^{-1}\left(\sum^M_{m=1}\sum^{\mu_m}_{k=0} L(d\phi_m)\partial_{\theta_m}\sigma_{m,k}\left(x,\frac{\phi_m(x)}{\eps}\right)r_{m,k}\right)+\eps^0\left(\ldots\right)+\ldots.
\end{align}
Observe that since each $r_{m,k}$ belongs to $\textrm{Ker }L(d\phi_m)$, the term of order $\frac{1}{\eps}$ in \eqref{new66} is in fact zero.

Equation \eqref{new63} is a special case of a substitution of the form
\begin{align}\label{new67}
u^a_\eps(x) = \cV^0\left(x,\frac{\phi_1(x)}{\eps},\ldots,\frac{\phi_M(x)}{\eps}\right),
\end{align}
for a function $\cV^0(x,\theta_1,\ldots,\theta_M)$ periodic in $(\theta_1,\ldots,\theta_M)$. One finds that solving \eqref{new53}(a) to order $\frac{1}{\eps}$ amounts to satisfying the condition
\begin{align}\label{new60}
\cL(\partial_\theta)\cV^0 :=\sum^M_{m=1}\tilde{L}(d\phi_m)\partial_{\theta_m} \cV^0= 0,
\end{align}
where $\tilde{L}(d\phi_m)=A^{-1}_d L(d\phi_m)$. Indeed, \eqref{new60} is satisfied if $\cV^0$ has the form
\begin{align}\label{new65}
\cV^0(x,\theta_1,\ldots,\theta_M)=\uv(x)+\sum^M_{m=1}\sum^{\mu_m}_{k=1}\sigma_{m,k}(x,\theta_m)r_{m,k},
\end{align}
which is consistent with the ansatz of \eqref{new63}.

One can try to solve \eqref{new53}(a) to higher order by constructing a \emph{corrected approximate solution} $u^c_\eps$ which, upon replacing $u^a_\eps$ in \eqref{new66}, results in the vanishing of the coefficients of higher powers of $\eps$. In fact, we will obtain more conditions on $\uv$ and the $\sigma_{m,k}$ which will ultimately determine our ansatz for $u^a_\eps$, by considering the existence of such a $u^c_\eps$ which corrects $u^a_\eps$ and agrees in the leading term. Agreement is meant in the sense that, for some $M$-periodic $\cV^1(x,\theta_1,\ldots,\theta_M)$, we have
\begin{align}\label{new12}
u^c_\eps(x)=u^a_\eps(x)+\eps \cV^1\left(x,\frac{\phi_1(x)}{\eps},\ldots,\frac{\phi_M(x)}{\eps}\right)=\cV^0\left(x,\frac{\phi_1(x)}{\eps},\ldots,\frac{\phi_M(x)}{\eps}\right)+\eps \cV^1\left(x,\frac{\phi_1(x)}{\eps},\ldots,\frac{\phi_M(x)}{\eps}\right).
\end{align}
Here $\cV^1$ is not generally expected to share the form of $\cV^0$ seen in \eqref{new65}, as we do not expect $\cV^1$ to solve \eqref{new60}. The function $\cV^1$ is determined by other differential equations arising from setting higher order terms of \eqref{new66} to zero.

\begin{rem}\label{new76}
\textup{
Ensuring the possibility of constructing such $u^c_\eps$, which suggests we have the right leading term $u^a_\eps$, requires \emph{solvability conditions} to hold; these will be translated into conditions known as the \emph{profile equations} on the \emph{profiles}, that is, $\uv(x)$ and the $\sigma_{m,k}(x,\theta_m)$. We will discuss the profile equations in greater depth later.
}
\end{rem}

Expansions with elliptic boundary layers were constructed in \cite{williams1}, treating the semilinear problem for generic $\beta$,\footnote{In fact, \cite{williams1} also treats $\beta$ resulting in glancing modes of order two, which do not result in blow-up, unlike higher-order glancing cases.} and in \cite{lescarret}, where a semilinear dispersive problem with maximally dissipative boundary conditions was considered. In both papers, higher order expansions in ascending powers of $\eps$, such as
\begin{align}
\cV^0\left(x,\frac{\phi_1(x)}{\eps},\ldots,\frac{\phi_M(x)}{\eps}\right)+\ldots+\eps^N \cV^N\left(x,\frac{\phi_1(x)}{\eps},\ldots,\frac{\phi_M(x)}{\eps}\right),
\end{align}
are used to justify the approximate solution as well as construct the exact solution.  High order expansions involving  surface waves, which arise from a failure of the uniform Lopatinski condition at a frequency $\beta$ in the elliptic region, were constructed and justified in \cite{marcou} for quasilinear hyperbolic systems.  Here in the spirit of \cite{cgw} we justify a leading term expansion for solutions to \eqref{new51} oscillating with multiple phases. In this situation it is impossible to construct a high order expansion without a small divisor assumption, an assumption we do not make.  While the small divisor assumption would exclude $\beta$ only from a set of measure zero, verifying it for a given $\beta$ can be difficult if not impossible.  

\subsubsection{The solution of the singular system and the main theorem}
\label{testlabel}

\emph{\quad} The following theorem verifies that the approximate solution is in fact close to the exact solution for small $\eps$.

\begin{theo}\label{maintheo}
Suppose the assumptions hold that $L(\partial_x)$ is hyperbolic with characteristics of constant multiplicity, the boundary $\{x_d=0\}$ is noncharacteristic, $(L(\partial_x),B(0))$ satisfies the uniform stability condition, and $\beta$ is a regular boundary frequency (i.e. Assumptions \ref{assumption1}, \ref{assumption2}, \ref{assumption3}, \ref{assumption4}, resp.)
Then given the exact solution $u_\eps$ of \eqref{new53} defined on a time interval $(-\infty,T]$ independent of $\eps$, for $u^a_\eps$ as in \eqref{new63}, where $\uv$ and the $\sigma_{m,k}$ satisfy\footnote{The profile equations are not solved exactly. Description of the sense in which error terms for the profile equations are to be small can be found in the error analysis. Essentially, the error must satisfy the hypotheses of Proposition \ref{new3}.} the profile equations, defined on a time interval $(-\infty,T']$ independent of $\eps$, we have 
\begin{align}
|u_\eps(x)-u^a_\eps(x)|_{L^\infty((-\infty,T_0]\times \overline{\bR}^d_+)}\rightarrow 0 \textrm{ as } \eps\rightarrow 0,
\end{align}
where $T_0$ is the minimum of $T$ and $T'$.
\end{theo}

Observe for this result that we need in particular a positive lower bound on the existence time $T_\eps$ of the exact solution $u_\eps$ to \eqref{new53} which holds uniformly for small $\eps$. That is, for some positive $\eps_0$,
\begin{align}
\inf_{(0,\eps_0]}T_\eps>0.
\end{align}
However, applying the standard theory to the problem \eqref{new53} yields existence times $T_\eps$ which shrink to zero as the Sobolev norms of the initial data blow up in the limit $\eps\rightarrow 0$.
To get estimates uniform in $\eps$, one may use a reformulation of the system \eqref{new53} known as the \emph{singular system}, which is also used in the framework of \cite{cgw}. In fact, to prove the main theorem, we show a stronger result involving the solution of the singular system. The singular system is obtained by recasting \eqref{new53} in terms of an unknown $U_\eps(x,\theta_0)$ periodic in $\theta_0$ such that a solution of \eqref{new53} is to be formed by making the substitution
\begin{align}\label{new56}
u_\eps(x)=U_\eps\left(x,\frac{\beta\cdot x'}{\eps}\right).
\end{align}
The singular system is written in the form
\begin{align}\label{new61}
\begin{split}
&(a)\;\partial_{x_d}U_\epsilon+\sum^{d-1}_{j=0}\tA_j(\epsilon U_\epsilon)\left(\partial_{x_j}+\frac{\beta_j \partial_{\theta_0}}{\epsilon}\right)U_\epsilon=F(\eps U_\epsilon)U_\epsilon,\\
&(b)\;B(\epsilon U_\epsilon)(U_\epsilon)|_{x_d=0}=G(x',\theta_0),\\
&(c)\;U_\epsilon=0 \text{ in } t<0,
\end{split}
\end{align}
where $F=A^{-1}_d \cF$. Indeed, the main benefit of expressing the problem in terms of the new unknown $U_\eps$ is that it is possible to obtain estimates uniform in $\eps$ for $U_\eps$ and thus existence of the solution $u_\eps$ to $\eqref{new53}$ in a fixed time interval independent of $\eps$, demonstrated in \cite{williams2} by M. Williams, where he constructed the exact solution and proved these estimates which are crucial to our analysis as well as the analysis of \cite{cgw}. The solution $U_\eps$ of the singular system is constructed in Theorem 7.1 of \cite{williams2}, with the use of the following iteration scheme:
\begin{align}\label{e5}
\begin{split}
&a)\;\partial_{x_d}U^{n+1}_\epsilon+\sum^{d-1}_{j=0}\tilde{A}_j(\eps U^n_\eps)\left(\partial_{x_j}+\frac{\beta_j \partial_{\theta_0}}{\epsilon}\right)U^{n+1}_\epsilon=F(\eps U^n_\epsilon)U^n_\epsilon,\\
&b)\;B(\epsilon U^{n}_\epsilon)(U^{n+1}_\epsilon)|_{x_d=0}=G(x',\theta_0),\\
&c)\;U^{n+1}_\epsilon=0 \text{ in } t<0.
\end{split}
\end{align}
Once the $U^{n}_\epsilon$ are obtained, $U_\eps$ is found by taking the limit as $n\to\infty$ of the $U^{n}_\epsilon$.

The stronger result we will show in order to get Theorem \ref{maintheo} is the following:
\begin{theo}\label{maintheo2}
Define
\begin{align}\label{new69}
\cU^0_\eps(x,\theta_0)=\cV^0\left(x,\theta_0+\uomega_1\frac{x_d}{\eps},\ldots,\theta_0+\uomega_M\frac{ x_d}{\eps}\right)
\end{align}
where $\cV^0$ is as in \eqref{new65}. If $G(x',\theta_0)\in H^{s+1}_T$ for sufficiently large $s$, then for the exact solution $U_\eps$ of $\eqref{new56}$, we have
\emph{
\begin{align}
|U_\eps(x,\theta_0)-\cU^0_\eps(x,\theta_0)|_{L^\infty(x_d,H^{s-1}(x',\theta_0))}\rightarrow 0 \textrm{ as } \eps\rightarrow 0.
\end{align}
}

\end{theo}

To handle the supremum norms taken over $x_d\in[0,\infty)$ and to accommodate the appearances of $\frac{x_d}{\eps}$ in the $M$ periodic arguments of \eqref{new69}, we introduce the placeholder $\xi_d=\frac{x_d}{\eps}$. Thus, we consider a function
\begin{align}\label{new70}
\cU^0(x,\theta_0,\xi_d):=\cV^0(x,\theta_0+\uomega_1 \xi_d,\ldots, \theta_0+\uomega_M \xi_d)=\uv(x)+\sum^M_{m=1}\sum^{\mu_m}_{k=1}\sigma_{m,k}(x,\theta_0+\uomega_m \xi_d)r_{m,k},
\end{align}
which we may also write as
\begin{align}\label{new16}
\cU^0(x,\theta_0,\xi_d)=\uv(x)+\sum^M_{m=1}\sum^{\mu_m}_{k=1}\psi_{m,k}(x,\theta_0,\xi_d)r_{m,k},
\end{align}
where profiles $\psi_{m,k}(x,\theta_0,\xi_d)$ are given in terms of the \emph{periodic profiles} $\sigma_{m,k}(x,\theta_m)$ by
\begin{align}\label{new21}
\psi_{m,k}(x,\theta_0,\xi_d):=\sigma_{m,k}(x,\theta_0+\uomega_m \xi_d).
\end{align}
As a result, assertions such as Lemma \ref{e11} and Proposition \ref{new3}, which concern substitutions of the form
\begin{align}
\cU^0_\epsilon(x,\theta_0)=\cU^0(x,\theta_0,\xi_d)|_{\xi_d=\frac{x_d}{\eps}},
\end{align}
play an important role in the error analysis.
\begin{rem}\label{new54}
\textup{
When $\uomega_m$ is real, since $\sigma_{m,k}(x,\theta_m)$ is periodic in $\theta_m$, the profile $\psi_{m,k}(x,\theta_0,\xi_d)$ defined in \eqref{new21} is almost-periodic in $(\theta_0,\xi_d)$. Furthermore, if $\beta$ is in the \emph{hyperbolic region} defined in Remark \ref{modeProps}(iii), all the $\uomega_m$ are real, and so in that case $\cU^0(x,\theta_0,\xi_d)$ is almost-periodic in $(\theta_0,\xi_d)$. In the case that some of the $\uomega_m$ are nonreal, in order to make sense of the substitutions \eqref{new21} we must first extend the $\sigma_{m,k}(x,\theta_m)$ to be defined for $\theta_m$ in the complex plane. This process is detailed in Section \ref{RoleOfNonreality}.
}
\end{rem}

\subsection{Role of nonreal phases and the resulting boundary layers}

\emph{\quad}We now point out a key feature distinguishing our study from that in \cite{cgw}. In \cite{cgw} it is assumed that $\beta$ belongs to the \emph{hyperbolic region}. This is equivalent to requiring that both (i) $\beta$ is a \emph{regular boundary frequency} (i.e. $\beta$ satisfies Assumption \ref{assumption4}) and (ii) all the eigenvalues $\uomega_m$ of the matrix $\frac{1}{i}\cA(\beta)$ are \emph{real}. In our study, we allow for any regular boundary frequency $\beta$ and must thus handle cases in which some of the $\uomega_m$ are nonreal. Hence, functions of complex variables must be considered in, for example, \eqref{new63}, \eqref{new65}, \eqref{new69}, and \eqref{new70}.

\subsubsection{Hyperbolic and elliptic profiles, and the emergence of the elliptic boundary layer}\label{RoleOfNonreality}

\emph{\quad} For each $m$ such that $\uomega_m$ is real, we call the $\sigma_{m,k}(x,\theta_m)$, $k=1,\ldots,\mu_m$, \emph{hyperbolic profiles}, and we note that in order to make the substitution $\theta_m=\theta_0+\uomega_m \xi_d$, such as those made in \eqref{new70}, a hyperbolic profile needs only to be defined for real $\theta_m$. On the other hand, when $\uomega_m$ is nonreal, we call the $\sigma_{m,k}(x,\theta_m)$ \emph{elliptic profiles}. For each elliptic profile, as we vary the parameters $\theta_0\in\bR$, $\xi_d\geq 0$, the value $\theta_m=\theta_0+\uomega_m \xi_d$ varies throughout one of the half complex planes $\{\textrm{Im } \theta_m\geq 0\}$, $\{\textrm{Im } \theta_m\leq 0\}$, depending on the sign of $\textrm{Im }\uomega_m$. To make sense of substituting $\theta_m=\theta_0+\uomega_m \xi_d$ into the argument of an elliptic profile, we define the profile first for real $\theta_m$ and then holomorphically extend the $\theta_m$-dependence into the appropriate half complex plane. To do so, we make use of the Fourier expansions of the profiles.

Consider a profile $\sigma_{m,k}(x,\theta_m)$, periodic in $\theta_m$ with mean $0$, and its expansion of the form
\begin{align}\label{new71}
\sigma_{m,k}(x,\theta_m)=\sum_{j\in \bZ\setminus 0}a_j(x)e^{ij\theta_m}.
\end{align}
For the moment we proceed formally, assuming that one can evaluate the above expression at $\theta_m=\theta_0+\uomega_m \xi_d$ by substituting $\theta_m=\theta_0+\uomega_m \xi_d$ in the argument of each exponential in the expansion for $\sigma_{m,k}(x,\theta_m)$ and yield a convergent expansion for the profile $\psi_{m,k}(x,\theta_0,\xi_d)$ as defined in \eqref{new21}. That is, given \eqref{new71}, we also have
\begin{align}\label{new31}
\psi_{m,k}(x,\theta_0, \xi_d)=\sum_{j\in \bZ\setminus 0}a_j(x)e^{ij(\theta_0+\uomega_m \xi_d)}.
\end{align}
We remark that, certainly, we have yet to make sense of this sum in the case that $\uomega_m\in\C\setminus\R$, but also that the space of convergence must be clarified  when $\uomega_m\in\R$. The convergence is made rigorous for both cases with Proposition \ref{propA}. Let us rewrite such an expansion in the following form:
\begin{align}\label{new72}
\psi_{m,k}(x,\theta_0,\xi_d)=\sigma_{m,k}(x,\theta_0+\uomega_m \xi_d)=\sum_{j\in \bZ\setminus 0}a_j(x)e^{ij\theta_0}e^{ij\textrm{Re}(\uomega_m) \xi_d}e^{-j\textrm{Im}(\uomega_m) \xi_d}.
\end{align}

Suppose $\sigma_{m,k}(x,\theta_m)$ is elliptic. We first consider the case with $\textrm{Im }\uomega_m > 0$. Then, the terms in the sum in \eqref{new72} with $j<0$ grow exponentially with $\xi_d$. It is easy to check that such terms result in an unsatisfactory candidate for our approximate solution $u^a_\eps(x)$, as defined in $\eqref{new63}$, which is nonphysical in the sense that it blows up in $L^\infty$ as $\eps\rightarrow 0 $. This leads us to construct $\sigma_{m,k}(x,\theta_m)$ such that $a_j(x)=0$ for $j<0$, and so an elliptic profile with $\textrm{Im }\uomega_m > 0$ is to have in fact the following expansion rather than the one in \eqref{new71}:
\begin{align}
\sigma_{m,k}(x,\theta_m)=\sum_{j\in \bZ^+\setminus 0}a_j(x)e^{ij\theta_m},\quad \textrm{for Im }\uomega_m>0.
\end{align}
Provided the above is a Fourier series in real $\theta_m$ converging in $H^{\frac{d}{2}+3}(x,\theta_m)$, one can show that it holomorphically extends into $\{\textrm{Im }\theta_m\geq 0\}$. Now, for the corresponding profile $\psi_{m,k}(x,\theta_0,\xi_d)=\sigma_{m,k}(x,\theta_0+\uomega_m \xi_d)$, we take \eqref{new31} and \eqref{new72}, apply the fact that $a_j(x)=0$ for $j<0$, and get
\begin{align}
\psi_{m,k}(x,\theta_0,\xi_d)=\sum_{j\in \bZ^+\setminus 0}a_j(x)e^{ij(\theta_0+\uomega_m \xi_d)}=\sum_{j\in \bZ^+\setminus 0}a_j(x)e^{ij\theta_0}e^{ij\textrm{Re}(\uomega_m) \xi_d}e^{-j\textrm{Im}(\uomega_m) \xi_d},\quad \textrm{for Im }\uomega_m>0.
\end{align}
As a result, the profile $\psi_{m,k}(x,\theta_0,\xi_d)$ must decay exponentially in $\xi_d$.

In the case that $\textrm{Im }\uomega_m < 0$, similar considerations lead us to construct a corresponding elliptic profile of the form
\begin{align}
\sigma_{m,k}(x,\theta_m)=\sum_{j\in \bZ^-\setminus 0}a_j(x)e^{ij\theta_m},\quad \textrm{for Im }\uomega_m<0,
\end{align}
a sum which holomorphically extends into $\{\textrm{Im }\theta_m\leq 0\}$. This kind of elliptic profile also results in $\psi_{m,k}(x,\theta_0,\xi_d)=\sigma_{m,k}(x,\theta_0+\uomega_m \xi_d)$ which decays exponentially as $\xi_d$ increases.

For a hyperbolic profile $\sigma_{m,k}(x,\theta_m)$, which has $\textrm{Im }\uomega_m = 0$, we do not make such restrictions on its coefficients $a_j(x)$, so we describe it with \eqref{new71} and $\psi_{m,k}(x,\theta_0,\xi_d)$ with the expansions \eqref{new31} and \eqref{new72}.

With the following remark, we summarize the different forms of expansions for profiles.
\begin{rem}
\textup{
For $Z_m=\{n\in\Z: n \textrm{Im}\uomega_m\geq0\}$, i.e.
\begin{align}
Z_m=\left\{ 
\begin{matrix}
  \mathbb{Z} \textrm{ if } \uomega_m\in\R, \\
  \mathbb{Z}^+ \textrm{ if } \textrm{Im }\uomega_m>0, \\
  \mathbb{Z}^- \textrm{ if } \textrm{Im }\uomega_m<0,
 \end{matrix}
\right.
\end{align}
each periodic profile has an expansion of the form
\begin{align}
\sigma_{m,k}(x,\theta_m)=\sum_{j\in Z_m\setminus 0}a_j(x)e^{ij\theta_m},
\end{align}
and each profile $\psi_{m,k}(x,\theta_0,\xi_d)$ has
\begin{align}
\psi_{m,k}(x,\theta_0,\xi_d)=\sum_{j\in Z_m\setminus 0}a_j(x)e^{ij(\theta_0+\uomega_m \xi_d)}=\sum_{j\in Z_m\setminus 0}a_j(x)e^{ij\theta_0}e^{ij\textrm{Re}(\uomega_m) \xi_d}e^{-j\textrm{Im}(\uomega_m) \xi_d}.
\end{align}
}
\end{rem}

Since the elliptic profiles $\sigma_{m,k}(x,\theta_m)$ result in $\psi_{m,k}(x,\theta_0,\xi_d)$ which decay in $\xi_d$, upon replacing the placeholder $\xi_d$ with $\frac{x_d}{\eps}$ we see they contribute to a boundary layer with rapid exponential decay in $x_d$ for small $\eps$, i.e. the \emph{elliptic boundary layer}, appearing in the approximate solution described by \eqref{new63}.

It appears that this paper is the first to rigorously justify leading order expansions involving multiple real and nonreal phase functions $\phi_m(x)$, and thus both hyperbolic and elliptic profiles, in quasilinear hyperbolic boundary problems.

\subsubsection{Loss of almost-periodicity and construction of the corrector}

\emph{\quad} In \cite{cgw}, the authors constructed leading order expansions similar to ours, though aspects of the analysis do not work with nonreal phases. As discussed in Remark \ref{new54}, in the case studied in \cite{cgw}, all the $\uomega_m$ are real, a condition which generally yields $\cU^0(x,\theta_0,\xi_d)$ which is almost-periodic in $(\theta_0,\xi_d)$. Thus, in that study, rather than having to define a kind of convergence of infinite sums of the form \eqref{new72}, since the functions of interest were all almost-periodic, it sufficed to work in the space $\cP^s_T$, defined by taking the closure of the set of finite trigonometric polynomials in $\cE^s_T$, defined by
\begin{align}
\cE^s_T=\{\cU(x,\theta_0,\xi_d):\sup_{\xi_d\geq 0}|\cU(\cdot,\cdot,\xi_d)|_{E^s_T} < \infty\},
\end{align}
where
\begin{align}\label{firstDefESpace}
E^s_T=C(x_d,H^s_T(x',\theta_0))\cap L^2(x_d,H^{s+1}_T(x',\theta_0)).
\end{align}
However, in the case that some of the $\uomega_m$ are nonreal, one loses the almost-periodicity of $\cU^0(x,\theta_0,\xi_d)$. The introduction of exponential decay in $\xi_d$ of the corresponding $\psi_{m,k}(x,\theta_0,\xi_d)$ presents serious obstacles to straightforwardly adopting the approach of \cite{cgw}. Many estimates proved in \cite{cgw} appear to have no analogue in our study.

As an example, we consider the construction of a projector used in \cite{cgw}, which we denote by $\bP$. The operator $\bP$ is first defined on finite polynomials of the form
\begin{align}\label{new50}
H(x,\theta_0,\xi_d)=\sum_{\kappa=(\kappa_0,\kappa_d)\in \bZ\times \bR} H_\kappa (x)e^{i\kappa_0 \theta_0+i \kappa_d \xi_d},
\end{align}
and is constructed such that $(\bP H)(x,\theta_0,\xi_d)=0$ if and only if there exists a solution $\cU$ of
\begin{align}
\cL'(\partial_{\theta_0},\partial_{\xi_d})\cU=H(x,\theta_0,\xi_d),
\end{align}
where $\cL'(\partial_{\theta_0},\partial_{\xi_d})=\cA(\beta)\partial_{\theta_0}+\partial_{\xi_d}$. Now suppose that, in an attempt to account for the existence of nonreal eigenvalues $\uomega_m$ of $\frac{1}{i}\cA(\beta)$, we take $H$ to be a finite polynomial of the form
\begin{align}\label{new49}
H(x,\theta_0,\xi_d)=\sum_{\kappa=(\kappa_0,\kappa_d)\in \bZ\times \{z\in\C:\textrm{Im }z\geq 0\}} H_\kappa (x)e^{i\kappa_0 \theta_0+i \kappa_d \xi_d},
\end{align}
noting the decay of $H$ in $\xi_d$. Given the existence of more than two nonreal eigenvalues of $\frac{1}{i}\cA(\beta)$, we found that a linear operator $\bP$ satisfying the same for such $H$ in any sensible function space\footnote{The nonlinearity of \eqref{new51} forces us to consider the action of $\cL'(\partial_{\theta_0},\partial_{\xi_d})$ (and thus $\bP$) on vectors such as $r e^{i\kappa_d\xi_d}\in \bR^N$ for all $\kappa_d$ in some set closed under $\Z$-linear combinations; in the event that more than two elliptic modes exist which are pairwise independent over $\bQ$, it follows there exists a sequence $r e^{i\kappa^n_d\xi_d}$ in the kernel of $\bP$ converging in $L^\infty$ to an element not in the kernel.} \emph{typically is unbounded}. On the other hand, the assumptions in \cite{cgw} lead to having no such eigenvalues, and in that case the projector $\bP$ can indeed be continuously extended to $\cP^s_T$. In fact, $\bP:\cP^s_T\rightarrow \cP^s_T$ is used in \cite{cgw} to obtain a corrector term for the expansion which is crucial to the error analysis of \cite{cgw}. We also need a corrector but we construct it in a different way.

To obtain a satisfactory corrector term, we found there was no need to define a projector on the space of almost-periodic profiles, such as $\bP$. In fact, we were able to dispense entirely with almost-periodic formulation of the profile equations, such as the one used in \cite{cgw}. Instead, it was sufficient to solve profile equations and use projectors in just the space for \emph{$M$-periodic} functions, which contains the $\cV^0(x,\theta_1,\ldots,\theta_M)$ and $\cV^1(x,\theta_1,\ldots,\theta_M)$ featured in \eqref{new65} and \eqref{new12}, namely $H^s(\overline{\bR}^{d+1}_+\times \bT^M)$, and therein we found we could naturally phrase much of the error analysis and construct a useful corrector term. In much of the error analysis, in place of the almost-periodic profiles with expansions in terms similar to \eqref{new50}, we deal with corresponding Fourier expansions in terms $e^{i\alpha\cdot\theta}$, $\alpha\in Z^M$ (see Definition \ref{indexdefn},) $\theta=(\theta_1,\ldots,\theta_M)$, holomorphically extended in $\theta$ to a selected subset $\bold{C}^M$ of $\C^M$ (see Remark \ref{new28},) circumventing the use of almost-periodicity of \cite{cgw}. The key observation was that for the construction of an appropriate correcter in $H^s(\overline{\bR}^{d+1}_+\times \bT^M)$, hyperbolic objects as well as elliptic objects could be handled simultaneously, as the terms $e^{i\alpha\cdot\theta}$ behave no worse for $\theta\in \bold{C}^M\subset \bC^M$ than they do at $\theta\in\bR^M$.

Interestingly, working with projectors only in $H^s(\overline{\bR}^{d+1}_+\times \bT^M)$ to obtain the corrector appeared to be significantly simpler than attempting to do so in a space similar to the one used in \cite{cgw}. For example, the continuity of our projectors is almost immediate,\footnote{See $\bE$ and $\bE^\flat$ appearing in Definition \ref{new47} and Remark \ref{new46}.} while even in the case treated in \cite{cgw}, the proof of continuity of the projector $\bP$ takes some effort. The trade-off for the gained simplicity is that we must describe convergence of expansions such as \eqref{new72} carefully with Proposition \ref{propA}, instead of just using a space which is a closure of trigonometric polynomials and assuming all the functions needed are almost-periodic. Finally, while we do as much of the error analysis in $H^s(\overline{\bR}^{d+1}_+\times \bT^M)$ as we can, we are forced to introduce the substitutions $\theta=(\theta_0+\uomega_1\xi_d,\ldots,\theta_0+\uomega_M\xi_d)$ and $\xi_d=\frac{x_d}{\eps}$ to use \cite{williams2} for the error analysis, with Proposition \ref{e9}.

\subsubsection{Complex transport equations and resonances}

\emph{\quad}In contrast with the real transport equations of the profile equations for the hyperbolic profiles, the complex transport equations for the elliptic profiles are generally not solvable. In \cite{williams1}, a Taylor expansion in $x_d$ is developed to approximately solve the corresponding complex transport equations of the profile equations, but using the usual Taylor error bounds would require us to consider $C^k$ regularity in place of $H^s$ regularity, presenting difficulties to our strategy for leading-order justification. Interestingly, however, in our proof of Theorem \ref{e1} we were able to show that merely requiring the complex transport equations of the profile equations to hold to first order at $x_d=0$ is sufficient for achieving $L^\infty$ convergence. The error from the complex transport equations belongs to a class of functions in $H^s_T(\overline{\R}^{d+1}_+,\bT^M)$ which are \emph{elliptically polarized}.\footnote{These functions satisfy $\bE_e \cV=\cV$, where $\bE_e$ is the \emph{elliptic projector} defined in \eqref{new14}.} This error term is thus handled by an application of Proposition \ref{new3}, which clarifies the sense in which such functions are small if they are zero at the boundary $\{x_d=0\}$.

Another unusual feature involving elliptic phases in the analysis is the handling of complex resonances, which result in couplings amongst the transport equations for the profiles. Careful examination of these resonances revealed that we could first work strictly within a subsystem of real transport equations \emph{not dependent on the elliptic profiles} and solve for the hyperbolic profiles. On the other hand, resonances between hyperbolic and elliptic phases do affect the complex transport equations for the elliptic profiles.

While, as in \cite{cgw}, hyperbolic profiles $\sigma_{m,k}$ are obtained as limits of $\sigma^n_{m,k}$ via an iteration scheme, there is no need to develop such a scheme to obtain the elliptic profiles. We construct satisfactory elliptic $\sigma_{m,k}$ by using finite propagation and regularity properties resulting from solving wave equations where time is thought of as the $x_d$ variable rather than the $t$ variable and using a trick of taking a frame of reference moving at the speed of propagation to ensure the $\sigma_{m,k}$ are supported in $t\geq 0$. The conditions which the elliptic profiles must satisfy at $x_d=0$ are straightforwardly read off from linear relations with the hyperbolic profiles at $x_d=0$ and $G$. Inserting the values of the profiles at $x_d=0$ into the complex transport equation evaluated on the boundary and requiring this to hold determines the values of the $x_d$-derivatives of the elliptic profiles on the boundary\footnote{For further discussion, see Section \ref{largeappsolution}.}. However, we still develop a sequence of elliptic `iterates' $\sigma^n_{m,k}$; we emphasize that while they do converge to the elliptic $\sigma_{m,k}$, these elliptic iterates are not used in the construction of the elliptic $\sigma_{m,k}$. The elliptic iterates are useful because they are convenient for the error analysis, allowing us to perform much of the analysis on the elliptic and hyperbolic parts simultaneously and in as uniform a way as possible.

\subsection{Use of simultaneous Picard iteration in the error analysis}

\emph{\quad}In the proof of Theorem 7.1 of \cite{williams2}, for some $T_0>0$, the iteration scheme for the singular system, \eqref{e5}, is used to produce $U_\eps(x,\theta_0)$ and iterates $U^n_\eps(x,\theta_0)$ bounded in $E^s_{T_0}$ uniformly with respect to $n$ and $\eps$ and which satisfy
\begin{align}\label{new73}
\lim_{n\rightarrow\infty} U^n_\eps = U_\eps \textrm{ in } E^{s-1}_{T_0} \textrm{ uniformly with respect to } \eps\in (0,\eps_0],
\end{align}
where $U_\eps$ solves the singular system \eqref{new61}.

In Section \ref{largehypsolution}, Proposition \ref{r2} of Section \ref{largeappsolution}, and Definition \ref{ansatzDefn} we construct a function $\cV^0(x,\theta)\in H^{s}(\overline{\bR}^{d+1}_+\times\mathbb{T}^M)$ which approximately satisfies the profile equations, and iterates $\cV^{0,n}(x,\theta)$ approximately satisfying the equations of a corresponding iteration scheme. The iterates $\cV^{0,n}$ are bounded in $\bH^{s+1}_{T_0}$ uniformly with respect to $n$ and satisfy
\begin{align}\label{new74}
\lim_{n\rightarrow\infty} \cV^{0,n} = \cV^0 \textrm{ in } \bH^{s}_{T_0}.
\end{align}
Again, we make the substitution seen in \eqref{new70}, plugging in $\theta=(\theta_0+\uomega_1 \xi_d,\ldots, \theta_0+\uomega_M \xi_d)$, to get
\begin{align}
\cU^{0,n}(x,\theta_0,\xi_d):=\cV^{0,n}(x,\theta_0+\uomega_1 \xi_d,\ldots, \theta_0+\uomega_M \xi_d),
\end{align}
followed by $\xi_d=\frac{x_d}{\eps}$, yielding
\begin{align}
\cU^{0,n}_\eps(x,\theta_0):=\cV^{0,n}\left(x,\theta_0+\uomega_1 \frac{x_d}{\eps},\ldots, \theta_0+\uomega_M \frac{x_d}{\eps}\right).
\end{align}
Thus, by using \eqref{new74} and applying the estimates given by Proposition \ref{propA} and Lemma \ref{e11}, we get that
\begin{align}
\lim_{n\to\infty}\cU^{0,n}_{\eps}= \cU^0_{\eps}\text{ in }   E^{s-1}_{T_0} \text{  uniformly  with respect to }\eps\in (0,\eps_0],
\end{align}
where $\cU^0_\eps$ is as in Theorem \ref{maintheo2}. Therefore, to conclude $\lim_{\eps\to 0}\cU^0_\eps(x,\theta_0)-U_\eps(x,\theta_0)=0$ in $E^{s-1}_{T_0}$ and finish the proof of Theorem \ref{maintheo2}, it is sufficient to show
\begin{align}\label{new75}
\lim_{\eps\to 0}|\cU^{0,n}_\eps-U^n_\eps|_{E^{s-1}_{T_0}}=0 \textrm{ for all } n.
\end{align}

Indeed, \eqref{new75} is proved in Section \ref{erroranalysis} by induction on $n$. One might try to prove the statement in this way by applying the estimate for the linearized singular system of Proposition \ref{e9} to $\left(\cU^{0,n+1}_\eps-U^{n+1}_\eps\right)$. The problem with this is that if we take $\bA(\eps U^n_\eps)$, which we define to be the operator appearing in the left hand side of the equation for $U^{n+1}_\eps$, i.e. \eqref{e5}(a), and apply it to the difference $\left(\cU^{0,n+1}_\eps-U^{n+1}_\eps\right)$, the resulting quantity does not necessarily converge to zero as $\eps\to 0$. To illustrate this point, consider the relation of $\cU^0_\eps$ to $\cV^0$. Recall that $\cV^0$ is constructed\footnote{Strictly speaking, details on the construction of $\cV^0$ come later in the paper with the solution of the profile equations. Instead, this refers to imposing the condition that $\cV^0$ has the form \eqref{new65}.} to satisfy $\eqref{new60}$ so that in \eqref{new66} the order $\frac{1}{\eps}$ terms are annihilated, noting that the terms of orders $\eps$, $\eps^2$, $\ldots$ shrink to zero with $\eps$, but observe that we have yet to handle the $O(1)$ terms. Similarly, plugging in $\cU^{0,n+1}_\eps$, an approximation for $\cU^0_\eps$, in place of $U^{n+1}_\eps$ in the equation for $U^{n+1}_\eps$ results in an $O(1)$ error.

This could be handled if a corrector for $\cV^{0,n+1}$ could be constructed, say $\cV^1$, analogous to the $\cV^1$ appearing in \eqref{new12}, which resulted in the elimination of the $O(1)$ terms upon the replacement of $\cV^{0,n+1}$ by $\cV^{0,n+1}+\eps \cV^1$ in a similar equation. Then one obtains a corrector for $\cU^{0,n+1}_\eps$, the function $\cU^1_\eps(x,\theta_0)$, from $\cV^1(x,\theta)$ by plugging in $\theta=(\theta_0+\uomega_1 \xi_d,\ldots,\theta_0+\uomega_M\xi_d)$ and $\xi_d=\frac{x_d}{\eps}$ as before. In fact, applying $\bA(\eps \cU^{0,n}_\eps)$ to $\left(\cU^{0,n+1}_\eps+\eps \cU^1_\eps-U^{n+1}_\eps\right)$ results in an error which converges to zero as $\eps\to 0$ given the existence of suitable $\cV^1$. After an application of the estimate from \cite{williams2}, \eqref{new75} can then be obtained, assuming $\eps \cU^1_\eps$ is bounded by $C \eps$ in an appropriate norm.

However, there are two major obstacles to constructing such a corrector $\cV^1$ for $\cV^{0,n+1}$. The first is that each of our iterates $\cV^{0,n+1}$ only \emph{approximately} solves the corresponding iterate equation of the profile equations. This is forced on us since our profile equations include complex transport equations which are not generally exactly solvable. On the other hand, as mentioned in Remark \ref{new76}, typically in geometric optics, the solvability conditions on an ansatz such as $\cV^0$ which imply $\cV^0$ has a corrector are satisfied by \emph{exactly} solving profile equations; similarly, in \cite{cgw}, solvability is achieved through exact solutions of the iterated profile equations. Instead, however, we found it was sufficient to solve each of the iterated profile equations up to an error which is zero at the boundary $\{x_d=0\}$ and `purely elliptic' in the sense that it depends only on components $\theta_m$ which are the arguments of elliptic profiles and is consequently \emph{elliptically polarized}; these conditions on the error are described precisely with the hypotheses of Proposition \ref{new3}. Essentially, we form new solvability conditions which are the iterated profile equations modified by including the error terms.

The second issue is that, since we do not make a small-divisor assumption, we can only guarantee solvability if we are working with \emph{finite} trigonometric polynomials\footnote{For details on solvability, see Proposition \ref{new7}.}, as opposed to general elements of $H^s_{T_0}$ with infinitely many nonzero Fourier coefficients. To resolve this, we approximate each $\cV^{0,n+1}$ by a finite trigonometric polynomial $\cV^{0,n+1}_p$, construct a corresponding corrector $\cV^1_p$, and use Proposition \ref{propA} and Lemma \ref{e11} to work with finite sums $\cU^{0,n+1}_p$ and $\cU^{0,n+1}_{p,\eps}$ approximating $\cU^{0,n+1}$ and $\cU^{0,n+1}_\eps$. Using these approximations, we are finally able to conclude \eqref{new75}.

\section{Profile equations: formulation with periodic profiles }\label{defns}

\begin{definition}\label{indexdefn}
(i) We define $Z^M\subset\mathbb{Z}^M$ by
\begin{align}
Z^M:=\{\alpha=(\alpha_i)^M_{i=1}:\alpha_i\in Z_i \},
\end{align}
where $Z_i$ is defined by
\begin{align}
Z_i:=\left\{ 
\begin{matrix}
  \mathbb{Z} \textrm{ for } i\in \mathcal{I} \cup \mathcal{O}, \\
  \mathbb{Z}^+ \textrm{ for } i\in \mathcal{P}, \\
  \mathbb{Z}^- \textrm{ for } i\in \mathcal{N},
 \end{matrix}
\right.
\end{align}
where $\cI$, $\cO$, $\cP$, and $\cN$ are as defined in the comments following Definition \ref{def2}. We also have the equivalent formulation \emph{$Z_i=\{n\in\Z: n \textrm{Im}\uomega_i\geq0\}$}.\\
(ii) We also define
\emph{
\begin{align}
Z^{M;k}:=\{\alpha\in Z^M:\textrm{ at most k components of }\alpha\textrm{ are nonzero}\}.
\end{align}
}
\end{definition}
\begin{definition} For $k=1,\,2$, we define the following spaces: \\
\begin{align}\label{a1}
H^{s;k}(\overline{\R}^{d+1}_+\times\mathbb{T}^M)=\left\{\mathcal{V}(x,\theta)\in H^s(\overline{\R}^{d+1}_+\times\mathbb{T}^M)
:\mathcal{V}(x,\theta)=\sum_{\alpha\in Z^{M;k}}V_\alpha(x)e^{i\alpha\cdot\theta} \right\}.
\end{align}

\end{definition}
For $s>(d+1+2)/2$, it is clear that multiplication defines a continuous map
\begin{align}
H^{s;1}(\overline{\R}^{d+1}_+\times\mathbb{T}^M)\times H^{s;1}(\overline{\R}^{d+1}_+\times\mathbb{T}^M)\to H^{s;2}(\overline{\R}^{d+1}_+\times\mathbb{T}^M).
\end{align}
\begin{rem}\label{new28}
\emph{
We define
\begin{align}
\bold{C}^M:=\bold{C}_1\times \bold{C}_2 \times \cdots \times \bold{C}_M,
\end{align}
where $\bold{C}_i$ is defined by
\begin{align}
\bold{C}_i:=\left\{
\begin{matrix}
  \R \textrm{ for } i\in \mathcal{I} \cup \mathcal{O}, \\
  \{\textrm{Im } z\geq 0\} \textrm{ for } i\in \mathcal{P}, \\
  \{\textrm{Im } z\leq 0\} \textrm{ for } i\in \mathcal{N}.
 \end{matrix}
\right.
\end{align}
For $\cV\in H^{s;2}(\overline{\R}^{d+1}_+\times\mathbb{T}^M)$ where $s>\frac{d}{2}+3$, one can show that $\textrm{spec }\mathcal{V}\subset Z^{M;2}$ implies $\mathcal{V}$ extends holomorphically in $\theta$ to the interior of $\bold{C}^M$. In particular, this uses the fact that then $\textrm{Im}(\alpha_i\theta_i) \geq 0$ for $i\in \mathcal{P}\cup \mathcal{N}$. This allows us to make sense of
\begin{align}
\mathcal{U}(x,\theta_0,\xi_d):=\mathcal{V}(x,\theta_0+\uomega_1\xi_d,\dots,\theta_0+\uomega_M\xi_d)
\end{align}
for $\theta_0\in\bT$, $\xi_d\geq 0$.
}
\end{rem}

\subsection{The periodic ansatz and its corrector. }\label{peransatzwithcor}

Our periodic ansatz will have the form
\begin{align}\label{new22}
\mathcal{V}^0(x,\theta)=\underline{v}(x)+\sum^M_{m=1}\sum^{\mu_m}_{k=1}\sigma_{m,k}(x,\theta_m)r_{m,k},
\end{align}
a function in $H^{s;1}(\overline{\R}^{d+1}_+\times\mathbb{T}^M)$, where $s$ is to be specified, and which is holomorphic in $\theta\in\bold{C}^M\subset\C^M$, in particular.
Using the notation
\begin{align}\label{new15}
\theta(\theta_0,\xi_d)=(\theta_0+\uomega_1\xi_d,\dots,\theta_0+\uomega_M\xi_d),
\end{align}
a calculation shows
\begin{align}
u^a_\epsilon(x):=\mathcal{U}^0(x,\theta_0,\xi_d)|_{\theta_0=\frac{\phi_0}{\epsilon},\xi_d=\frac{x_d}{\epsilon}}=\mathcal{V}^0(x,\theta(\theta_0,\xi_d))|_{\theta_0=\frac{\phi_0}{\epsilon},\xi_d=\frac{x_d}{\epsilon}}
\end{align}
results in the vanishing of the terms of order $\frac{1}{\epsilon}$ when plugging $u^a_\epsilon$ into $P(\epsilon u_\epsilon,\partial_x)u_\epsilon$ in \eqref{new53}(a). If we plug in a corrected approximate solution
\begin{align}
u^c_\epsilon(x):=\left(\mathcal{U}^0(x,\theta_0,\xi_d)+\epsilon \mathcal{U}^1(x,\theta_0,\xi_d)\right)|_{\theta_0=\frac{\phi_0}{\epsilon},\xi_d=\frac{x_d}{\epsilon}}=\left(\mathcal{V}^0(x,\theta(\theta_0,\xi_d))+\epsilon\mathcal{V}^1(x,\theta(\theta_0,\xi_d))\right)|_{\theta_0=\frac{\phi_0}{\epsilon},\xi_d=\frac{x_d}{\epsilon}}
\end{align}
where we do not necessarily have $\mathcal{V}^1$ of the form \eqref{new22}, but require $\mathcal{V}^1(x,\theta)\in H^{s;2}(x,\theta)$, then the terms of order $\epsilon^0$ cancel out if and only if
\begin{align}\label{zeroterms}
\mathcal{L}'(\partial_{\theta_0},\partial_{\xi_d})\mathcal{U}^1+\tilde{L}(\partial_x)\mathcal{U}^0+\mathcal{M}'(\mathcal{U}^0)\partial_{\theta_0}\mathcal{U}^0=F(0)\mathcal{U}^0,
\end{align}
using the notation
\begin{align}
\cL'(\partial_{\theta_0},\partial_{\xi_d}):=\cA(\beta)\partial_{\theta_0}+\partial_{\xi_d},
\end{align}
\begin{align}
\cM'(\cU):=\sum^{d-1}_{j=0}\partial_u \tilde{A}_j (0)\cU \beta_j.
\end{align}
A sufficient condition for \eqref{zeroterms} is
\begin{align}\label{new23}
\mathcal{L}(\partial_\theta)\mathcal{V}^1+\tilde{L}(\partial_x)\mathcal{V}^0+\mathcal{M}(\mathcal{V}^0)\partial_\theta\mathcal{V}^0=F(0)\mathcal{V}^0,
\end{align}
where
\begin{align}
\cL(\D_\theta)=\sum^M_{m=1}\tilde{L}(d\phi_m)\D_{\theta_m}
\end{align}
and
\begin{align}
\mathcal{M}(\cV)\partial_\theta := \sum^M_{m=1}\sum^{d-1}_{j=0} \beta_j \partial_u \tilde{A}_j(0)\cV\partial_{\theta_m}.
\end{align}
Observe that the operator $\cL(\partial_\theta)$ is singular, so that one cannot simply invert it to solve for $\cV^1$ in \eqref{new23}. One approach is to ensure the existence of a solution $\cV^1$ by imposing the following condition on $\cV^0$:
\begin{align}\label{new23b}
\tilde{L}(\partial_x)\mathcal{V}^0+\mathcal{M}(\mathcal{V}^0)\partial_\theta\mathcal{V}^0-F(0)\mathcal{V}^0\in \textrm{Im }\mathcal{L}(\partial_\theta),
\end{align}
a condition which turns out to be equivalent to a differential equation\footnote{This differential equation is \eqref{new26}.} in $\cV^0$.

The condition \eqref{new23b} on $\cV^0$ does not clearly determine $\uv$ and the $\sigma_{m,k}$, the parts of $\cV^0$ which remain to be chosen.
Strictly speaking, we will not satisfy \eqref{new23} or \eqref{new23b}. Instead, we replace the differential equation (equivalent to \eqref{new23b}) with another one\footnote{This is the differential equation \eqref{new27}.} which agrees when we plug in $\theta=\theta(\theta_0,\xi_d)$ and can be decomposed into the system of profile equations (in $\uv$ and the $\sigma_{m,k}$.) This allows us to solve for the profiles and a corrector $\cV^1$ satisfying an equation which agrees with \eqref{new23} on $\theta=\theta(\theta_0,\xi_d)$. Thus, this condition is \emph{also} sufficient for \eqref{zeroterms}.

Now we work towards defining the projection operators to appear in these differential equations, $\bE$ and $\bE^\flat$, which will allow us to construct the ansatz $\cV^0$ and its corrector $\cV^1$.

\begin{definition}

Setting $\phi:=(\phi_1,\dots,\phi_M)$, we call $\alpha\in Z^{M;2}$ a \emph{characteristic mode} provided $\textrm{\emph{det}}  L(d(\alpha \cdot \phi))=0$ and write $\alpha\in \cC$. We decompose
\begin{align}\label{a19}
\cC=\cup^M_{m=1}\cC_m,\textrm{ where }\cC_m=\{\alpha\in Z^{M;2} : \alpha \cdot \phi = n_\alpha\phi_m \textrm{ for some }n_\alpha\in\mathbb{Z} \}.
\end{align}
\end{definition}
\begin{rem}
\emph{
Normally, the above is defined with $\mathbb{Z}^{M;2}$ in place of $Z^{M;2}$, but in preventing our solution from blowing up,  we have restricted its spectrum such that only $\alpha\in Z^{M;2}$ are considered.
}
\end{rem}

Now we are ready to define the projectors which will give us the profile equations and the differential equation for the condition \eqref{new23b}.

\begin{definition}\label{new47}
Let $\mathcal{V}\in H^{s+1;2}_T$. The action of  $\mathbb{E}$ on $\mathcal{V}$ is defined by
\textup{
\begin{align}\label{a20}
\mathbb{E}=\mathbb{E}_0+\sum^M_{m=1}\mathbb{E}_m,\textrm{ where }\mathbb{E}_0\mathcal{V}=V_0\textrm{ and }\mathbb{E}_m\mathcal{V}=\sum_{\alpha\in \cC_m\setminus0}P_m V_\alpha(x)e^{in_\alpha\theta_m},
\end{align}
}
where $P_m$ denotes the projection onto $\textrm{Ker }L(d\phi_m)$, the action of $\mathbb{E}^\flat$ on $\mathcal{V}$ is defined by
\textup{
\begin{align}\label{a21}
\mathbb{E}^\flat=\mathbb{E}^\flat_0+\sum^M_{m=1}\mathbb{E}^\flat_m,\textrm{ where }\mathbb{E}^\flat_0\mathcal{V}=V_0\textrm{ and }\mathbb{E}^\flat_m\mathcal{V}=\sum_{\alpha\in \cC_m\setminus0}P_m V_\alpha(x)e^{i\alpha\cdot\theta},
\end{align}
}
and we use the notation
\begin{align}\label{new14}
\bE_h=\bE_0+\sum_{m\in\mathcal{I}\cup\mathcal{O}}\bE_m,\quad \bE_e=\sum_{m\in\mathcal{P}\cup\mathcal{N}}\bE_m.
\end{align}
\end{definition}

\begin{rem}\label{new46}
\textup{
(i) It is shown in \cite{cgw} that $\mathbb{E}$ is a continuous map in the $H^s$ norm. Here, to complete the statement $\mathbb{E}:H^{s;2}\rightarrow H^{s;1}$, we must also have for each $\alpha\in \cC_m$ that $n_\alpha\in Z_m$. This is easily verified from the definitions of $Z_m$, $\cC_m$, and the $n_\alpha$. A similar calculation is done in Remark \ref{r1}. (ii) It is not hard to show the continuity of $\mathbb{E}^\flat:H^{s;2}\rightarrow H^{s;2}$.
}
\end{rem}
\begin{rem}\label{new18}
\textup{
Denoting evaluation at $\theta=\theta(\theta_0,\xi_d)$ by $\Phi$, observe $\Phi\circ\mathbb{E}=\Phi\circ\mathbb{E}^\flat$. The projector $\mathbb{E}$ serves as the tool for solving for $\uv(x)$ and the periodic profiles $\sigma_{m,k}(x,\theta_m)$ of our ansatz. The projector $\mathbb{E}^\flat$ is key to describing solvability (such as the condition \eqref{new23b}) and is thus used in the error analysis in solving away an $O(1)$ error output of $\Phi\circ(I-\mathbb{E})$ written in the form $\Phi\circ(I-\mathbb{E}^\flat)$, so that we can prove Theorem \ref{e1}. That is the purpose which the following proposition serves.
}
\end{rem}
\begin{prop}\label{new7}
Given $\mathcal{H}\in H^{s;2}_T(x,\theta)$ with finitely many nonzero Fourier coefficients, suppose
\begin{align}\label{new9}
\bE^\flat{\mathcal{H}}=0,
\end{align}
Then there exists $\mathcal{V}\in H^{s;2}_T(x,\theta)$ such that
\begin{align}\label{new8}
\mathcal{L}(\partial_\theta)\mathcal{V}=\mathcal{H}.
\end{align}
\end{prop}
\begin{proof}
We write out the series of $\mathcal{H}$ as
\begin{align}
\mathcal{H}(x,\theta)=\sum_{\alpha\in Z^{M;2}}H_\alpha(x)e^{i\alpha\cdot\theta}.
\end{align}
Now we search for $V_\alpha(x)$ such that, upon defining
\begin{align}\label{new10}
\mathcal{V}(x,\theta):=\sum_{\alpha\in Z^{M;2}}V_\alpha(x)e^{i\alpha\cdot\theta},
\end{align}
one has
\begin{align}
\mathcal{L}(\partial_\theta)\mathcal{V}(x,\theta)=\sum^M_{m=1}\sum_{\alpha\in Z^{M;2}}i\alpha_m\tilde{L}(d\phi_m)V_\alpha(x)e^{i\alpha\cdot\theta}=\sum_{\alpha\in Z^{M;2}}H_\alpha(x)e^{i\alpha\cdot\theta}.
\end{align}
This holds if and only if for all $\alpha\in Z^{M;2}$
\begin{align}\label{new11}
i\tilde{L}\left(\sum_m \alpha_m \beta, \alpha\cdot \omega\right)V_\alpha(x)=H_\alpha(x),
\end{align}
where $\omega=(\uomega_1,\ldots,\uomega_M)$. We handle first the case where $\alpha\in \cC_j\setminus 0$ for some $j=1,2,\ldots,M$. Note that
\begin{align}
0=\mathbb{E}^\flat_j \mathcal{H}(x,\theta)
=\sum_{\alpha\in \cC_j\setminus0}P_j H_\alpha(x)e^{i\alpha\cdot\theta},
\end{align}
which implies for each $\alpha\in \cC_j$
\begin{align}
0=P_j H_\alpha(x).
\end{align}
Recalling from Lemma \ref{lem2} that $\textrm{Im }A^{-1}_d L(d\phi_j)=\textrm{Ker }P_j$, we see there exists $W_\alpha(x)\in H^s_T(x)$ such that
\begin{align}
H_\alpha(x)=i\tilde{L}(d\phi_j)W_\alpha(x).
\end{align}
Observe then
\begin{eqnarray}
H_\alpha(x)
&=& i\tilde{L}(n_\alpha\beta,n_\alpha\uomega_j)\frac{1}{n_\alpha}W_\alpha(x),\\
&=& i\tilde{L}\left(\sum^M_{m=1}\alpha_m\beta,\alpha\cdot \omega\right)\frac{1}{n_\alpha}W_\alpha(x).
\end{eqnarray}
So we may satisfy \eqref{new11} by defining $V_\alpha(x)=\frac{1}{n_\alpha}W_\alpha(x)$.\\
For the case $\alpha=0$, where $\bE^\flat_0 \mathcal{H}=H_\alpha=0$, or for any other $\alpha$ such that $H_\alpha=0$, we may satisfy \eqref{new11} by taking $V_\alpha=0$.\\
For $\alpha\notin \cC$ with $H_\alpha\neq 0$, we have $\textrm{det }L(\sum_m \alpha_m\beta,\alpha\cdot\omega)\neq 0$ by definition of $\cC$, so then \eqref{new11} can be solved directly for $V_\alpha(x)$. Since there are only finitely many nonzero $H_\alpha$, we find the same holds for the $V_\alpha$, and so the sum in \eqref{new10} indeed describes an element of $H^{s;2}_T(x,\theta)$.
\end{proof}

\begin{rem}\label{new25}
(i)\textup{ In view of Proposition \ref{new7}, in order to construct a corrector $\cV^1$ appropriate for our periodic ansatz $\cV^0$ to satisfy \eqref{new23}, it is tempting to require something like
\begin{align}\label{new26}
\bE^\flat(\tilde{L}(\partial_x)\mathcal{V}^0+\mathcal{M}(\mathcal{V}^0)\partial_\theta\mathcal{V}^0)=\bE^\flat(F(0)\mathcal{V}^0).
\end{align}
However, note that to use Proposition \ref{new7}, if the $\mathcal{V}^0$ we seek has infinitely many nonzero Fourier coefficients, \eqref{new26} alone is insufficient; we must use finite trigonometric polynomial approximations. Furthermore, even if $\mathcal{V}^0$ is replaced with a finite trigonometric polynomial approximation, there is no clear way to solve \eqref{new26} for $\mathcal{V}^0$ of the form \eqref{new22}. We explain: the components of the projector $\bE^\flat$ are the $\bE^\flat_m$; typically, one obtains some profile, say, $\sigma_{m,k}(x,\theta_m)$ in \eqref{new22} by examining the $m$th component of a projector. However, $\bE^\flat_m$ maps into $H^{s;2}(x,\theta)$, meaning its output generally varies with \emph{more than one} component of $\theta$ as opposed to varying with only the $\theta_m$ component. Meanwhile, $\bE$ has components $\bE_m$ mapping into $H^s(x,\theta_m)$, spaces suited to $\sigma_{m,k}(x,\theta_m)$. In fact, $\cV^0$ having the form \eqref{new22} is equivalent to having
\begin{align}\label{new29}
\bE\cV^0=\cV^0.
\end{align}
}
(ii)\textup{ With the considerations made in Remark \ref{new18}, it is natural to require instead of \eqref{new26} that
\begin{align}\label{new27}
\bE(\tilde{L}(\partial_x)\mathcal{V}^0+\mathcal{M}(\mathcal{V}^0)\partial_\theta\mathcal{V}^0)=\bE(F(0)\mathcal{V}^0).
\end{align}
In fact, \eqref{new27} gives the solvability conditions discussed in Remark \ref{new76} which ensure the existence of a corrected approximate solution $u^c_\eps(x)$ resulting in the vanishing of the terms of order $O(1)$ in \eqref{new53}(a). To see this, recall that the $O(1)$ terms vanish if and only if \eqref{zeroterms} holds. If we satisfy \eqref{new27}, we have
\begin{align}\label{reagt}
(I-\bE)(\tilde{L}(\partial_x)\mathcal{V}^0+\mathcal{M}(\mathcal{V}^0)\partial_\theta\mathcal{V}^0-F(0)\cV^0)=\tilde{L}(\partial_x)\mathcal{V}^0+\mathcal{M}(\mathcal{V}^0)\partial_\theta\mathcal{V}^0-F(0)\cV^0.
\end{align}
Once $\cV^0$ is determined by imposing \eqref{new27}, with Proposition \ref{new7} one can solve for a corrector $\cV^1$ in
\begin{align}\label{reagtflat}
\cL(\partial_\theta)\cV^1=-(I-\bE^\flat)(\tilde{L}(\partial_x)\mathcal{V}^0+\mathcal{M}(\mathcal{V}^0)\partial_\theta\mathcal{V}^0-F(0)\cV^0),
\end{align}
(if we assume for the moment that we are working with finite trigonometric polynomials.) Although such $\cV^1$ does not necessarily satisfy \eqref{new23}, after evaluating at $\theta=\theta(\theta_0,\xi_d)$, which, recall, is denoted by $\Phi$, it follows from \eqref{reagt} and \eqref{reagtflat} that
\begin{align}\label{trueCondition}
\Phi\left(\cL(\partial_\theta)\cV^1\right)=-\Phi\left(\tilde{L}(\partial_x)\mathcal{V}^0+\mathcal{M}(\mathcal{V}^0)\partial_\theta\mathcal{V}^0-F(0)\cV^0\right),
\end{align}
where we have used the property $\Phi\circ(I-\E^\flat)=\Phi\circ(I-\E)$. In fact, for $\cU^1=\Phi(\cV^1)$ and $\cU^0=\Phi(\cV^0)$, \eqref{trueCondition} is equivalent to \eqref{zeroterms}, giving $u^c_\eps(x)$ solving \eqref{new53}(a) to order $O(1)$.
}
\end{rem}

It is from the components of \eqref{new27} that we will obtain the components $\uv$, $\sigma_{m,k}$, $m=1,\ldots,M$, $k=1,\ldots,\mu_m$, of $\cV^0$. Now we collect the set of equations which determines $\cV^0(x,\theta)\in H^{s;1}(\overline{\R}^{d+1}_+\times\bT^M)$ for $s$ sufficiently large (specified later):
\begin{align}\label{fullConditions}
\begin{split}
&a)\;\bE\cV^0=\cV^0\\
&b)\;\bE\left(\tilde{L}(\partial_x)\cV^0+\cM(\cV^0)\partial_{\theta}\cV^0\right)=\bE(F(0)\cV^0)\text{ in }x_d\geq 0\\
&c)\;B(0)\cV^0(x',0,\theta_0,\dots,\theta_0)=G(x',\theta_0)\\
&d)\;\cV^0=0\text{ in }t<0.
\end{split}
\end{align}

In summary, \eqref{fullConditions}(a) gives that $\cL(\partial_\theta)\cV^0=0$, and as a result, in the system for the original solution as a perturbation, this guarantees that \eqref{new53}(a) is solved to order $O(\frac{1}{\eps})$. Equation \eqref{fullConditions}(b) represents the solvability conditions discussed in Remark \ref{new76}, guaranteeing the existence of a corrected approximate solution $u^c_\eps(x)$ as seen in \eqref{new12} which solves \eqref{new53}(a) to order $O(1)$. Finally, equations \eqref{fullConditions}(c) and \eqref{fullConditions}(d) straightforwardly correspond to \eqref{new53}(b) and \eqref{new53}(c).

\textbf{Resonances. }

The next major task is to rephrase \eqref{fullConditions}(b) in terms of $\uv$ and the $\sigma_{m,k}$ to get the interior profile equations. \eqref{fullConditions}(b) will be broken down into components where $\bE$ is replaced by $E_0$ and $\bE_p$, $p=1,\ldots,M$, yielding equations for $\uv$ and the $\sigma_{p,l}$, $l=1,\ldots,\mu_p$. The nonlinearity in \eqref{fullConditions}(b) leads to resonances, which are generated in products such as $\sigma_{q,k}(x,\frac{\phi_q}{\eps})\D_{\theta_r}\sigma_{r,k'}(x,\frac{\phi_r}{\eps})$ when there exists a relation of the form
\begin{align}
n_p \phi_p = n_q \phi_q + n_r \phi_r\textrm{ where }p\in\{1,\ldots,M\}\setminus\{q,r\}\textrm{ and }(p,q,r)\in\Z\times Z_q\times Z_r.
\end{align}
This means that $\phi_q$ oscillations interact with $\phi_r$ relations to produce $\phi_p$ oscillations. While such $\phi_p$ oscillations arise on their own in products such as $\sigma_{q,k}(x,\frac{\phi_q}{\eps})\D_{\theta_r}\sigma_{r,k'}(x,\frac{\phi_r}{\eps})$, not even some $\theta_p$ dependence is present in the product $\sigma_{q,k}(x,\theta_q)\D_{\theta_r}\sigma_{r,k'}(x,\theta_r)$ before we evaluate $\theta_p=\frac{\phi_p}{\eps},\,\theta_q=\frac{\phi_q}{\eps},\,\theta_r=\frac{\phi_r}{\eps}$. Meanwhile $\E_p$ accounts for these interactions when it maps such $(\theta_q,\theta_r)$ oscillations to $\theta_p$ oscillations, as it maps $H^{s;2}(x,\theta)$ to $H^s(x,\theta_p)$. When we express $\E_p$\eqref{fullConditions}(b) solely in terms of functions of $(x,\theta_p)$, it takes some effort to evaluate $\E_p(\cM(\cV^0)\partial_{\theta}\cV^0)$ in terms of the profiles, in particular to describe these interactions in terms of $\sigma_{q,k}$ and $\sigma_{r,k'}$. This is the purpose of the interaction integrals, \eqref{c4}, appearing in the profile equations. To handle these difficulties we introduce some new definitions regarding resonances and interaction integrals.
\begin{definition}
We say $(\phi_p,\phi_q,\phi_r)$ is an \emph{ordered triple of resonant phases} and that $(\phi_q,\phi_r)$ forms a \emph{$p$-resonance} if
\begin{align}\label{rename1}
n_p\phi_p=n_q\phi_q+n_r\phi_r
\end{align}
for some $(n_p,n_q,n_r)\in \Z\times Z_q \times Z_r$, each entry nonzero. The triple $(\phi_p,\phi_q,\phi_r)$ is called \emph{normalized} if we have both (i) $gcd(n_p,n_q,n_r)=1$ and (ii) if $p\in \mathcal{I}\cup\mathcal{O}$, then $n_p>0$.
\end{definition}
First, we explicitly treat the case that the only normalized triples of resonant phases are permutations of $(\phi_p,\phi_q,\phi_r)$ corresponding to one of the six rearrangements of \eqref{rename1}. It follows from the next remark that the sign of $n_p$ is determined, so this normalization uniquely determines $n_p$, $n_q$, $n_r$.
\begin{rem}\label{r1}
\emph{
For a triple of resonant phases $(\phi_p,\phi_q,\phi_r)$ with \eqref{rename1}, one has $(n_p,n_q,n_r)\in Z_p \times Z_q \times Z_r$. To see this, taking the imaginary part of $n_p\phi_p=n_q\phi_q+n_r\phi_r$, observe
\begin{align}
n_p \textrm{Im }\uomega_p =n_q \textrm{Im }\uomega_q+n_r\textrm{Im }\uomega_r.
\end{align}
Thus, if $(\phi_q,\phi_r)$ forms a $p$-resonance, since $n_q\textrm{Im }\uomega_q$ and $n_r\textrm{Im }\uomega_r$ are nonnegative, so is $n_p \textrm{Im }\uomega_p$, so it follows that $n_p\in Z_p$.  
}
\end{rem}

We now make a related observation with the following proposition which will greatly simplify our solution of the interior equations.
\begin{prop}\label{new24}
(i) A hyperbolic resonance (i.e. a $p$-resonance, where $p\in \mathcal{I}\cup \mathcal{O}$) can only be formed by a pair of hyperbolic phases, and 
(ii) an elliptic resonance can only be formed by a pair including at least one elliptic phase.
\end{prop}

\begin{proof}
(i) Suppose for some $p\in\mathcal{I}\cup\mathcal{O}$ that $n_p\phi_p=n_q\phi_q+n_r\phi_r$ with $(n_p,n_q,n_r)\in Z_p\times Z_q\times Z_r$. Considering the imaginary part gives
\begin{align}
0=n_q \textrm{Im }\uomega_q + n_r \textrm{Im }\uomega_r,
\end{align}
so recalling $n_i\in Z_i$, $n_i\textrm{Im }\uomega_i\geq0$, we see $\textrm{Im }\uomega_q=\textrm{Im }\uomega_r=0$.  Hence $q,r\in \mathcal{I}\cup \mathcal{O}$. The proof of (ii) is similar.
\end{proof}

\subsection{The large system for individual profiles. }\label{largesystem}
Recall the decomposition of the projector
\begin{align}
\bE=\bE_0+\sum^M_{m=1}\bE_m
\end{align}
For $m=1,\ldots,M$, we enumerate by $\{\ell_{m,k},k=1,\ldots,\nu_m\}$ a basis of vectors for the left eigenspace of
\begin{align}\label{b16a}
i\cA(\beta)=A_d^{-1}(0)(\utau I+\sum_{j=1}^{d-1}A_j(0)\ueta_j)
\end{align}
associated to the eigenvalue $-\uomega_m$, and with the following property:
\begin{align}\label{b17}
\ell_{m,k}\cdot r_{m',k'}=\begin{cases}1, \text{ if }m=m'\text{ and }k=k'\\0, \text{ otherwise}\end{cases}.
\end{align}
For $v\in\bC^N$ set
\begin{align}\label{b18}
P_{m,k}v=(\ell_{m,k}\cdot v)r_{m,k} \text{ (no complex conjugation here) }.
\end{align}
So now we may write
\begin{align}\label{new42}
\bE=\bE_0+\sum^M_{m=1}\sum^{\mu_m}_{k=1}\bE_{m,k},
\end{align}
where
\begin{align}\label{b20}
\bE_{m,k}(V_\alpha e^{i\alpha\cdot\theta}):=\begin{cases}(P_{m,k}V_\alpha) e^{in_\alpha \theta_m},\;\alpha\in\cC_m\setminus 0\\0, \text{ otherwise}\end{cases}; \text{ i.e., }\bE_{m,k}=P_{m,k}\bE_m.
\end{align}
Recall the potential use of the projectors $\bE_m$ discussed in Remark \ref{new25}. We will obtain a system of equations for $\uv(x)$, $\sigma_{m,k}(x,\theta_m)$ by applying $\bE_0$ and the $\bE_{m,k}$ to \eqref{fullConditions}(b).

We omit proof of the following lemma, which is almost identical to that of Lemma 2.11 of \cite{cgw}.
\begin{lem}\label{b21}
Suppose $\bE\cV^0=\cV^0$.  Then
\begin{align}\label{b22}
\bE_{m,k}(\tilde{L}(\partial_x)\cV^0)=(X_{\phi_m}\sigma_{m,k})r_{m,k}
\end{align}
where $X_{\phi_m}$ is the characteristic vector field associated to $\phi_m$:
\begin{align}\label{b23}
X_{\phi_m}:=\partial_{x_d}+\sum^{d-1}_{j=0}-\partial_{\xi_j}\uomega_m(\beta)\partial_{x_j}.
\end{align}

\end{lem}

 We proceed by defining the interaction integrals which will appear in the interior equations.
\begin{definition}
Suppose $(\phi_q,\phi_r)$ forms a normalized $p$-resonance, with
\begin{align}\label{hello}
n_p\phi_p=n_q\phi_q+n_r\phi_r.
\end{align}
For any $f\in H^s_T(x,\theta_q)$ define $f_{n_q}\in H^s_T(x,\theta_q)$ to be the image of $f$ under the \emph{preparation map}
\begin{align}\label{c3}
f(x,\theta_q)=\sum_{k\in\bZ}f_k(x)e^{ik\theta_q}\to \sum_{k\in\bZ}f_{kn_q}(x)e^{ikn_q\theta_q}.
\end{align}
Suppose $s>\frac{d+3}{2}+1$ and that $\sigma_{q,k}$, $\sigma_{r,k'}\in H^s_T(\overline{\bR}^{d+1}_+\times\mathbb{T})$. We get exactly two normalized $p$-resonance formations from the arrangements of \eqref{hello}:
\begin{align}
n_p\phi_p=n_q\phi_q+n_r\phi_r,\quad n_p\phi_p=n_r\phi_r+n_q\phi_q.
\end{align}
To these equations we associate, respectively, the two families of \emph{prepared integrals}:

\begin{align}\label{c4}
\begin{split}
&J^{k,k'}_{p,n_q,n_r}(x,\theta_p):=\\
&\quad\frac{1}{2\pi}\int^{2\pi}_0(\sigma_{q,k})_{n_q}\left(x,\frac{n_p}{n_q}\theta_p-\frac{n_r}{n_q}\theta_r\right)\partial_{\theta_r}\sigma_{r,k'}(x,\theta_r)d\theta_r,\;\;k\in\{1,\dots,\mu_q\}, k'\in\{1,\dots,\mu_r\},\\
&J^{k,k'}_{p,n_r,n_q}(x,\theta_p):=\\
&\quad\frac{1}{2\pi}\int^{2\pi}_0(\sigma_{r,k})_{n_r}\left(x,\frac{n_p}{n_r}\theta_p-\frac{n_q}{n_r}\theta_q\right)\partial_{\theta_q}\sigma_{q,k'}(x,\theta_q)d\theta_q,\;\;k\in\{1,\dots,\mu_r\}, k'\in\{1,\dots,\mu_q\}.
\end{split}
\end{align}
\end{definition}

\begin{rem}
\emph{
Strictly speaking, $J^{k,k'}_{p,n_q,n_r}$ and $J^{k,k'}_{p,n_r,n_q}$ are symbolically $I^{k,k'}_{n_q,n_p,n'_r}$ and $I^{k,k'}_{n'_r,-n_p,n_q}$ (as defined in \cite{cgw}, with respect to the triple $(\phi_q,\phi_p,\phi_r)$ and $n_q\phi_q=n_p\phi_p+n'_r\phi_r$)  where $n'_r=-n_r$.  The difference is that we do not require $n_q>0$, $q<p<r$ -- recall, instead, our concern is that $(n_p,n_q,n_r)\in Z_p \times Z_q \times Z_r$.
}
\end{rem}
The following proposition shows that the prepared integrals `pick out' the $p$-resonances.
\begin{prop}\label{rename2}
 Suppose $s>\frac{d+3}{2}+1$ and that $\sigma_{q,k}$, $\sigma_{r,k'}\in H^s_T(\overline{\bR}^{d+1}_+\times\mathbb{T})$ have Fourier series
\begin{align}\label{c5}
\sigma_{q,k}(x,\theta_q)=\sum_{j\in Z_q\setminus 0}a_j(x)e^{ij\theta_q}\text{ and }\sigma_{r,k'}(x,\theta_r)=\sum_{j\in Z_r\setminus 0}b_j(x)e^{ij\theta_r}.
\end{align}
Given $(\phi_q,\phi_r)$ forms a normalized $p$-resonance, the prepared integral $J_{p,n_q,n_r}^{k,k'}(x,\theta_p)$ belongs to $H^{s-1}_T(x,\theta_p)$ and has Fourier series
\begin{align}\label{c6}
J_{p,n_q,n_r}^{k,k'}(x,\theta_p)=\sum_{j\in\mathbb{Z}} a_{jn_q}(x)b_{jn_r}(x)i\cdot(jn_r)e^{ijn_p\theta_p}.
\end{align}
For the other integral in \eqref{c4}, one switches $n_q$ and $n_r$ above. Moreover, the functions $J_{p,n_q,n_r}^{k,k'}(x,\theta_p)$ can be described for $\theta_p\in\bold{C}_p$ through analytic extension of the expansions into the complex half plane $\bold{C}_p$ in the same manner that the profiles $\sigma_{m,k}(x,\theta_m)$ are extended into $\bold{C}_m$ in Section \ref{RoleOfNonreality}.
\end{prop}
\begin{rem}
\textup{We note that for $p\in\mathcal{P}\cup\mathcal{N}$, in \eqref{c6}, we can sum over just $j\in \mathbb{Z}^+\setminus 0$. To see this, observe that then $(\phi_q,\phi_r)$ forms an elliptic resonance, so Proposition \ref{new24} guarantees one of $\phi_q$, $\phi_r$ is elliptic -- without loss of generality, say $\phi_q$. Thus, $Z_q\setminus 0$ consists only of negative integers or positive integers, and the only nonzero Fourier coefficients of $\sigma_{q,k}$ are those $a_j(x)$ appearing in \eqref{c5}, with $j\in Z_q\setminus 0$; we have $a_j(x)=0$ for all other $j$. For example, if $q\in \cP$, then $Z_q=\bZ^+$, and so, noting $n_q>0$, we have the $a_{jn_q}(x)$ appearing in \eqref{c6} are only nonzero for $j>0$. Indeed, $q\in \cN$ implies the same of the $a_{jn_q}(x)$, so we may sum over just $j\in \mathbb{Z}^+\setminus 0$ in either case. We add that, for such $j$, the $e^{ijn_p\theta_p}$ factors are bounded as $\theta_p$ ranges over ${\bf C}_p$.}
\end{rem}
\begin{proof}[Proof of Proposition \ref{rename2}]
The main part of the proof is showing that the expansion \eqref{c6} converges to the desired result in $H^{s-1}_T(x,\theta_p)$, where we have restricted our attention to real $\theta_p$. The proof of this is very similar to the proof of Proposition 2.13 from \cite{cgw}. Analytic extension into the complex half plane $\bold{C}_p$ is done in the same way that the expansions for the profiles are extended into their corresponding complex half planes, as discussed in Section \ref{RoleOfNonreality}.
\end{proof}
\textbf{Interior equations. }

By applying $\E_0$ and the $\E_{m,k}$ to \eqref{fullConditions}(b), we obtain the system for $\uv(x)$ and the $\sigma_{m,k}(x,\theta_m)$ with the following proposition.
\begin{prop}\label{new35}
Suppose $(\phi_\uq,\phi_\ur)$ forms a normalized $\up$-resonance, with
\begin{align}\label{new39}
n_\up\phi_\up=n_\uq\phi_\uq+n_\ur\phi_\ur.
\end{align}
and that all other normalized triples of resonant phases are permutations of $(\phi_\up,\phi_\uq,\phi_\ur)$ each corresponding to one of the six rearrangements of \eqref{new39}.\\
(i) Given $\cV^0$ as in \eqref{new29} and represented by \eqref{new22}, the equation \eqref{fullConditions}(b) is equivalent to the following system:
\begin{align}\label{nc9}
\tilde L(\partial_x)\uv+\sum_{j=0}^{d-1}\sum_{m=1}^M\sum_{k,k'=1}^{\mu_m}\frac{1}{2\pi}\left(\int^{2\pi}_0\sigma_{m,k}(x,\theta_m)\partial_{\theta_m}\sigma_{m,k'}(x,\theta_m)d\theta_m\right) R_{j,m}^{k,k'}=F(0)\uv\; ;
\end{align}
\begin{align}\label{nc10}
\begin{split}
&(a)\; X_{\phi_p}\sigma_{p,l}(x,\theta_p)+\sum^{d-1}_{j=0}\sum^{\mu_p}_{k'=1}a_{p,l,j}^{k'}(\uv)\partial_{\theta_p}\sigma_{p,k'}(x,\theta_p)+\\
&(b)\;\sum^{d-1}_{j=0}\sum^{\mu_p}_{k=1}\sum^{\mu_p}_{k'=1}b_{p,l,j}^{k,k'}\sigma_{p,k}(x,\theta_p)\partial_{\theta_p}\sigma_{p,k'}(x,\theta_p)-\sum^{d-1}_{j=0}\sum^{\mu_p}_{k=1}\sum^{\mu_p}_{k'=1}b_{p,l,j}^{k,k'}\frac{1}{2\pi}\int^{2\pi}_0\sigma_{p,k}(x,\theta_p)\partial_{\theta_p}\sigma_{p,k'}(x,\theta_p)d\theta_p+\\
&(c)\; \qquad \sum^{d-1}_{j=0}J_{p,l,j}(x,\theta_p) = \sum^{\mu_p}_{k=1}e^k_{p,l}\sigma_{p,k}(x,\theta_p)\\
&\textrm{for } p\in \mathcal{I}\cup\mathcal{O},\;l\in \{1,\ldots,\mu_p\},
\end{split}
\end{align}
where
\begin{align}\label{new40}
J_{p,l,j}(x,\theta_p)= \begin{cases}
\displaystyle
  \sum^{\mu_\uq}_{k=1}\sum^{\mu_\ur}_{k'=1} c_{p,l,j}^{k,k'}\; J_{\up,n_\uq,n_\ur}^{k,k'}(x,\theta_p)+
\sum^{\mu_\ur}_{k=1}\sum^{\mu_\uq}_{k'=1} d_{p,l,j}^{k,k'}\; J_{\up,n_\ur,n_\uq}^{k,k'}(x,\theta_p),\; p=\up, \\
\displaystyle
  \sum^{\mu_\up}_{k=1}\sum^{\mu_\ur}_{k'=1} c_{p,l,j}^{k,k'}\; J_{\uq,-n_\up,n_\ur}^{k,k'}(x,\theta_p)+
\sum^{\mu_\ur}_{k=1}\sum^{\mu_\up}_{k'=1} d_{p,l,j}^{k,k'}\; J_{\uq,n_\ur,-n_\up}^{k,k'}(x,\theta_p),\; p=\uq, \\
\displaystyle
  \sum^{\mu_\uq}_{k=1}\sum^{\mu_\up}_{k'=1} c_{p,l,j}^{k,k'}\; J_{\ur,n_\uq,-n_\up}^{k,k'}(x,\theta_p)+
\sum^{\mu_\up}_{k=1}\sum^{\mu_\uq}_{k'=1} d_{p,l,j}^{k,k'}\; J_{\ur,-n_\up,n_\uq}^{k,k'}(x,\theta_p),\; p=\ur, \\
0,\; \textrm{otherwise};
 \end{cases}
\end{align}

\begin{align}\label{nc11}
\begin{split}
&(a)\; X_{\phi_p}\sigma_{p,l}(x,\theta_p)+\sum^{d-1}_{j=0}\sum^{\mu_p}_{k'=1}a_{p,l,j}^{k'}(\uv)\partial_{\theta_p}\sigma_{p,k'}(x,\theta_p)+\\
&(b)\; \sum^{d-1}_{j=0}\sum^{\mu_p}_{k=1}\sum^{\mu_p}_{k'=1}b_{p,l,j}^{k,k'}\sigma_{p,k}(x,\theta_p)\partial_{\theta_p}\sigma_{p,k'}(x,\theta_p)+\\
&(c)\; \qquad \sum^{d-1}_{j=0}J_{p,l,j}(x,\theta_p) = \sum^{\mu_p}_{k=1}e^k_{p,l}\sigma_{p,k}(x,\theta_p)\\
&\textrm{for }p\in \mathcal{P}\cup\mathcal{N},\;l\in \{1,\ldots,\mu_p\},
\end{split}
\end{align}
where
\begin{align}\label{new41}
J_{p,l,j}(x,\theta_p)= \begin{cases}
\displaystyle
  \sum^{\mu_\uq}_{k=1}\sum^{\mu_\ur}_{k'=1} c_{p,l,j}^{k,k'}\; J_{\up,n_\uq,n_\ur}^{k,k'}(x,\theta_p)+
\sum^{\mu_\ur}_{k=1}\sum^{\mu_\uq}_{k'=1} d_{p,l,j}^{k,k'}\; J_{\up,n_\ur,n_\uq}^{k,k'}(x,\theta_p),\; p=\up, \\
0,\; \textrm{otherwise}.
 \end{cases}
\end{align}
The constant vectors $R_{j,m}^{k,k'}$ appearing in \eqref{nc9} are given by
\begin{align}\label{nc12}
R_{j,m}^{k,k'}=\beta_j(\partial_u\tilde{A}_j(0)\cdot r_{m,k}) r_{m,k'}.
\end{align}
The constant scalars $b_{p,l,j}^{k,k'}$, $c_{p,l,j}^{k,k'}$, $d_{p,l,j}^{k,k'}$, $e^k_{p,l}$, and
the coefficients of the scalar linear function of $\uv$, $a_{p,l,j}^{k}(\uv)$, in \eqref{nc10}-\eqref{new41} are given by similar formulas,
but now involving dot products with the vector $\ell_{p,l}$.\\
(ii) Equations \eqref{nc9}, \eqref{nc10}, form a hyperbolic subsystem, that is, a system in $\underline{v}$ and $\sigma_{p,l}$ for $p\in\mathcal{I}\cup\mathcal{O}$ independent of $\sigma_{q,k}$ for all $q\in\mathcal{P}\cup\mathcal{N}$.
\end{prop}
\begin{rem}
\textup{While the hyperbolic interior equations are independent of the elliptic profiles, the elliptic interior equations \eqref{nc11} are not in general independent of the hyperbolic profiles, since these appear in the elliptic interaction integrals if some phase paired with a hyperbolic phase forms an elliptic resonance.}
\end{rem}

\begin{proof}[Proof of Proposition \ref{new35}]
Regarding (i), \eqref{nc9} follows from application of $\bE_0$ to \eqref{fullConditions}(b), and \eqref{nc10} from application of $\bE_{p,l}$ to \eqref{fullConditions}(b). One similarly arrives at \eqref{nc11}: in doing so, one finds that the second triple sum appearing in \eqref{nc10}(b) is zero for $p\in\mathcal{P}\cup\mathcal{N}$, since then $\textrm{spec }\sigma_{p,k}$, $\textrm{spec }\sigma_{p,k'}\subset Z_p\setminus 0$ which is either $\Z^+\setminus 0$ or $\Z^-\setminus 0$, resulting in \eqref{nc11}(b). The differences between \eqref{new40} and \eqref{new41} merely account for the fact that for $p\in\mathcal{I}\cup\mathcal{O}$, we have $n_p\in Z_p\setminus 0=\Z\setminus 0$ implies $-n_p\in Z_p\setminus 0=\Z\setminus 0$, while for $p\in\mathcal{P}\cup\mathcal{N}$, we have $n_p\in Z_p\setminus 0$ implies $-n_p\notin Z_p\setminus 0$. The reverse implication of (i) follows upon recalling the decomposition \eqref{new42}.\\
For proof of (ii), first we show the elliptic profiles drop out of \eqref{nc9}. This follows from the argument made in proving (i) that terms such as those in the triple sum of \eqref{nc9} with $m\in\mathcal{P}\cup\mathcal{N}$ are zero, since then $\textrm{spec }\sigma_{m,k}$, $\textrm{spec }\sigma_{m,k'}\subset Z_p\setminus 0$. Finally, to see that no $\sigma_{q,k}$ with $q\in\mathcal{P}\cup\mathcal{N}$ appears in \eqref{nc10}, we have only to check the interaction integral terms $J_{p,l,j}$ of \eqref{nc10}(c). However, the terms of \eqref{new40} defining $J_{p,l,j}$ involve only profiles corresponding to $p$-resonant phases; $J_{p,l,j}$ is only nonzero for some $p\in\mathcal{I}\cup\mathcal{O}$ if $p=\up$, $\uq$, or $\ur$ and, by Proposition \ref{new24}, $\up$, $\uq$, $\ur\in\mathcal{I}\cup\mathcal{O}$. Thus, in this case, no elliptic profiles appear in $J_{p,l,j}$.
\end{proof}

From the above, we find the hyperbolic subsystem has the same form as the system in \cite{cgw} and we will solve it similarly. On the other hand, we will not solve the elliptic interior equations exactly.

\textbf{Boundary equations. }

Now that we have the interior equations for the profiles, we formulate the boundary conditions to be imposed on the profiles. From \eqref{fullConditions}(c), we get
\begin{align}\label{finalBCs}
\begin{split}
&(a)\;B(0)\uv=\uG(x')\\
&(b)\;B(0)\cV^{0*}(x',0,\theta_0,\dots,\theta_0)=\\
&\qquad B(0)\left(\sum_{m\in\cI\cup\cP\cup\cN}\sum^{\nu_{k_m}}_{k=1}\sigma_{m,k}(x',0,\theta_0)r_{m,k}+\sum_{m\in\cO}\sum^{\nu_{k_m}}_{k=1}\sigma_{m,k}(x',0,\theta_0)r_{m,k}\right)=G^*(x',\theta_0),
\end{split}
\end{align}
where $\cV^{0*}$ denotes the mean zero part of the periodic function $\cV^0$, and we have similar for $G^*$. While \eqref{finalBCs}(a) gives boundary data for $\uv$, it remains to establish satisfactory boundary conditions on the individual profiles $\sigma_{m,k}$ from \eqref{finalBCs}(b). This is done with Proposition \ref{rename3}.

\begin{lem}\label{bases}
Each of the sets of vectors
\begin{align}\label{posSpecBasis}
\{B(0)r_{m,k}:m\in\cI\cup\cP,\,k=1,\ldots,\mu_m\},
\end{align}
\begin{align}\label{negSpecBasis}
\{B(0)r_{m,k}:m\in\cI\cup\cN,\,k=1,\ldots,\mu_m\},
\end{align}
is a basis of $\C^p$.
\end{lem}
\begin{proof}
First we prove \eqref{posSpecBasis} is a basis of $\C^p$. Recall our assumption of uniform stability, Assumption \eqref{assumption3}, which tells us that $B(0)$ maps a basis for the stable subspace $\E^s(\utau,\ueta)$ to a basis for $\C^p$. Thus, according to Lemma \ref{SubspaceDecompLem}, which states that
\begin{align}
\E^s(\utau,\ueta)=\oplus_{m\in\cI\cup\cP}\textrm{Ker }L(d\phi_m),
\end{align}
we have that the set \eqref{posSpecBasis} is a basis for $\C^p$. To see the same holds for \eqref{negSpecBasis}, recall from Remark \ref{modeProps}(ii) the fact that for each nonreal $\uomega_m$, say, without loss of generality, $m\in\cP$, there is another eigenvalue $\uomega_{m'}$ satisfying $\overline{\uomega}_{m'}=\uomega_m$, thus with $m'\in\cN$, and that similar holds for the associated eigenvectors. It follows that
\begin{eqnarray}
\{B(0)r_{m,k}:m\in\cI\cup\cN,\,k=1,\ldots,\mu_m\}&=&\{B(0)\overline{r}_{m,k}:m\in\cI\cup\cP,\,k=1,\ldots,\mu_m\}\\
&=&\{\overline{B(0)r_{m,k}}:m\in\cI\cup\cP,\,k=1,\ldots,\mu_m\}.\label{conjBasis}
\end{eqnarray}
Clearly, \eqref{conjBasis} spans the same subspace as \eqref{posSpecBasis}, i.e. $\C^p$, and so \eqref{negSpecBasis} is also a basis of $\C^p$.
\end{proof}
\begin{prop}\label{rename3}
The data $(\sigma_{m,k}(x',0,\theta_0);m\in \mathcal{I}\cup \mathcal{P} \cup \mathcal{N},k\in\{1,\dots,\mu_m\})$ are determined by the data  $(\sigma_{m,k}(x',0,\theta_0);m\in \mathcal{O},k\in\{1,\dots,\mu_m\})$ in that there exist constant matrices $\mathbb{M}^\pm$ such that the zero-mean boundary condition,
\begin{align}\label{new43}
B(0)\mathcal{V}^{0*}(x',0,\theta_0,\ldots,\theta_0)
=B(0)\left(\sum^M_{m=1}\sum^{\mu_m}_{k=1}\sigma_{m,k,n}(x',0,\theta_0)r_{m,k}\right)
=G^*(x',\theta_0),
\end{align}
is equivalent to the condition
\begin{align}\label{c10}
(\sigma^\pm_{m,k}(x',0,\theta_0);m\in \mathcal{I}\cup\mathcal{P}\cup\mathcal{N},k\in\{1,\dots,\mu_m\})
=\mathbb{M}^\pm(G^{*,\pm},\sigma^\pm_{m,k}(x',0,\theta_0);m\in \mathcal{O},k\in\{1,\dots,\mu_m\}),
\end{align}
(using $+$ $(-)$ to denote a part with positive (negative) spectrum.)
\end{prop}
\begin{proof}
We rewrite \eqref{new43} as
\begin{align}\label{new44}
\begin{split}
& B(0)\cV^{0*}(x',0,\theta_0,\dots,\theta_0)=\\
&\qquad B(0)\left(\sum_{m\in\cI\cup\cP\cup\cN}\sum^{\mu_m}_{k=1}\sigma_{m,k}(x',0,\theta_0)r_{m,k}+\sum_{m\in\cO}\sum^{\mu_m}_{k=1}\sigma_{m,k}(x',0,\theta_0)r_{m,k}\right)=G^*(x',\theta_0).
\end{split}
\end{align}
Recall $\mathcal{P}$-profiles have positive spectra and $\mathcal{N}$-profiles have negative spectra. Thus, using the subscript $n$ to denote the $n$th term in a Fourier series, we get for $n>0$,
\begin{align}\label{c8}
B(0)\left(\sum_{m\in \mathcal{I}\cup \mathcal{P}}\sum^{\mu_m}_{k=1}\sigma_{m,k,n}(x',0,\theta_0)r_{m,k}\right)=G^*_n(x',\theta_0)
-B(0)\left(\sum_{m\in \mathcal{O}}\sum^{\mu_m}_{k=1}\sigma_{m,k,n}(x',0,\theta_0)r_{m,k}\right),
\end{align}
and for $n<0$,
\begin{align}\label{c9}
B(0)\left(\sum_{m\in \mathcal{I}\cup \mathcal{N}}\sum^{\mu_m}_{k=1}\sigma_{m,k,n}(x',0,\theta_0)r_{m,k}\right)=G^*_n(x',\theta_0)
-B(0)\left(\sum_{m\in \mathcal{O}}\sum^{\mu_m}_{k=1}\sigma_{m,k,n}(x',0,\theta_0)r_{m,k}\right).
\end{align}
By Lemma \ref{bases}, both sets $\{B(0)r_{m,k}:k\in\{1,\dots,\mu_m\},m\in \mathcal{I} \cup \mathcal{P}\}$, $\{B(0)r_{m,k}:k\in\{1,\dots,\mu_m\},m\in \mathcal{I} \cup \mathcal{N}\}$ are bases for $\C^p$, and the desired result follows from this.
\end{proof}

So we now have a sub-system (\eqref{nc9}, \eqref{nc10}) in just the hyperbolic profiles $\sigma_{m,k}$, $m\in\cI\cup\cO$, and the mean $\uv$ with boundary data for $\uv$ and determination of $\mathcal{I}$-boundary data from $\mathcal{O}$-boundary data as in \eqref{finalBCs}(b), where \eqref{finalBCs}(b) is equivalent to \eqref{c10}.\\
\begin{rem}
\textup{We will seek $\uv$, $\sigma_{p,l}$ such that, in addition to the equations above, we satisfy the initial conditions
\begin{align}\label{r8}
\uv=0\;\textrm{ and }\sigma_{p,l}=0\textrm{ in }t\leq 0\textrm{ for all }p,l.
\end{align}
The large system consists of the interior equations \eqref{nc9}-\eqref{new41} with boundary equations \eqref{finalBCs} and the initial conditions \eqref{r8}.}
\end{rem}
\subsection{Solution of the hyperbolic sub-system of the large system. }\label{largehypsolution}

First we will obtain the `hyperbolic part' of our solution, $\cV^0_h$, which will satisfy, in place of \eqref{fullConditions},
\begin{align}\label{hypConditions}
\begin{split}
&a)\;\bE_h\cV^0_h=\cV^0_h\\
&b)\;\bE_h\left(\tilde{L}(\partial_x)\cV^0_h+\cM(\cV^0_h)\partial_{\theta}\cV^0_h\right)=\bE_h(F(0)\cV^0_h)\text{ in }x_d\geq 0\\
&c)\;B(0)\cV^0_h(x',0,\theta_0,\dots,\theta_0)=G(x',\theta_0)\\
&d)\;\cV^0_h=0\text{ in }t<0.
\end{split}
\end{align}
Condition \eqref{hypConditions}(a) serves the same purpose as \eqref{fullConditions}(a), but adds the restriction that $\cV^0_h$ only consists of the mean and the hyperbolic profiles. Thus, recalling the form for $\cV^0$, we see $\cV^0_h$ takes the form
\begin{align}\label{hypForm}
\cV^0_h(x,\theta_1,\ldots,\theta_M)=\uv(x)+\sum^M_{m\in\cI\cup\cO}\sum^{\mu_m}_{k=1}\sigma_{m,k}(x,\theta_m)r_{m,k}.
\end{align}
The condition \eqref{hypConditions}(b) is exactly the hyperbolic sub-system (\eqref{nc9},\eqref{nc10}) found in Proposition \ref{new35}. This follows directly from Proposition \ref{new35} and the definition of $\E_h$. Since $\cV^0_h$ has the form \eqref{hypForm}, the boundary condition \eqref{hypConditions}(c) becomes \eqref{c10}, ignoring the appearance of the elliptic profiles ($\sigma_{m,k}$ with $m\in\cP\cup\cN$) in the left hand side.

To solve this, we employ the same approach as that which is used in \cite{cgw} to solve the full system of profile equations. This is done by solving the system with an iteration scheme such as the following
\begin{align}\label{hypConditionsIt}
\begin{split}
&a)\;\bE_h\cV^{0,n}_h=\cV^{0,n}_h\\
&b)\;\bE_h\left(\tilde{L}(\partial_x)\cV^{0,n}_h+\cM(\cV^{0,n-1}_h)\partial_{\theta}\cV^{0,n}_h\right)=\bE_h(F(0)\cV^{0,n-1}_h)\text{ in }x_d\geq 0\\
&c)\;B(0)\cV^{0,n}_h(x',0,\theta_0,\dots,\theta_0)=G(x',\theta_0)\\
&d)\;\cV^{0,n}_h=0\text{ in }t<0,
\end{split}
\end{align}
where we note that \eqref{hypConditionsIt}(a) means $\cV^{0,n}_h$ is of the form
\begin{align}\label{hypFormIt}
\cV^{0,n}_h(x,\theta_1,\ldots,\theta_M)=\uv^n(x)+\sum^M_{m\in\cI\cup\cO}\sum^{\mu_m}_{k=1}\sigma^n_{m,k}(x,\theta_m)r_{m,k},
\end{align}
where iterates $\uv^n(x)$, $\sigma^n_{m,k}(x,\theta_m)$ for $m\in\cI\cup\cO$ satisfy an iterated version of the hyperbolic sub-system of the profile equations, (\eqref{nc9},\eqref{nc10}) which is encapsulated by \eqref{hypConditionsIt}(b).

The following proposition gives the solution to the iterated system as well as the solution to the hyperbolic sub-system itself. We omit the proofs of these as they are almost identical to the proof of Proposition 2.19 together with that of Proposition 2.21 from \cite{cgw}.
\begin{prop}\label{cgwprops}

Let $T>0$, $m>\frac{d+3}{2}+1$ and suppose that $G(x',\theta_0)\in H^m_T$.

(i)  Setting $\cV^{0,0}_h=0$, there exist unique iterates $\cV^{0,n}_h\in \bH^m_T$, $n\geq 1$, solving the system \eqref{hypConditionsIt}.

(ii) For some $0<T_0\leq T$ the system \eqref{hypConditions} has a unique solution $\cV^0_h\in \bH^m_{T_0}$. Furthermore,
\begin{align}
\lim_{n\to\infty}\cV^{0,n}_h = \cV^0_h\textrm{ in }\bH^{m-1}_{T_0},
\end{align}
and the traces $\cV^0_h|_{x_d=0}$, $\uv|_{x_d=0}$, and $\sigma_{p,l}|_{x_d=0}$, $p\in\cI\cup\cO$, all lie in $H^m_{T_0}$.
\end{prop}

\subsection{Approximate solution of the equations for the elliptic profiles. }\label{largeappsolution}

Recall from Proposition \ref{new35} the elliptic interior equations:
\begin{align}\label{ellipIntEqn}
\begin{split}
&(a)\; X_{\phi_p}\sigma_{p,l}(x,\theta_p)+\sum^{d-1}_{j=0}\sum^{\mu_p}_{k'=1}a_{p,l,j}^{k'}(\uv)\partial_{\theta_p}\sigma_{p,k'}(x,\theta_p)+\\
&(b)\; \sum^{d-1}_{j=0}\sum^{\mu_p}_{k=1}\sum^{\mu_p}_{k'=1}b_{p,l,j}^{k,k'}\sigma_{p,k}(x,\theta_p)\partial_{\theta_p}\sigma_{p,k'}(x,\theta_p)+\\
&(c)\; \qquad \sum^{d-1}_{j=0}J_{p,l,j}(x,\theta_p) = \sum^{\mu_p}_{k=1}e^k_{p,l}\sigma_{p,k}(x,\theta_p)\\
&\textrm{for }p\in \mathcal{P}\cup\mathcal{N},\;l\in \{1,\ldots,\mu_p\}.
\end{split}
\end{align}
These complex transport equations may not generally have exact solutions. Instead, our elliptic profiles $\sigma_{p,l}$ will approximately solve these equations in the sense that they will hold at the boundary $x_d=0$. It is sufficient to simply evaluate the above expression at $x_d=0$, isolate $\partial_{x_d}\sigma_{p,l}(x',0,\theta_p)$, and require that $\sigma_{p,l}$ has the appropriate $x_d$-derivative at the boundary. Of course, we will also have to adhere to the boundary conditions, so that the trace $\sigma_{p,l}(x',0,\theta_0)$ at the boundary satisfies \eqref{finalBCs}(b), and we must satisfy initial conditions \eqref{r8}. With Proposition \ref{r2}, we construct such elliptic profiles $\sigma_{p,l}(x,\theta_p)$ by solving some wave equations in which $x_d$ plays the role of the time variable and with appropriate boundary data corresponding to $\sigma_{p,l}(x',0,\theta_p)$ and $\partial_{x_d}\sigma_{p,l}(x',0,\theta_p)$.

In addition to the exact elliptic interior equations, we consider semilinear equations in iterates $\sigma^{n}_{p,l}$ for $p\in\mathcal{P}\cup\mathcal{N}$, $l\in\{1,\ldots,\mu_p\}$, almost identical to the iteration scheme which is used to obtain the hyperbolic profiles. While we do not need elliptic iterates $\sigma^{n}_{p,l}$ to get the elliptic profiles $\sigma_{p,l}$, which are instead obtained by the process described above, they allow us to handle hyperbolic parts and elliptic parts uniformly in the error analysis. The semilinear equations in the elliptic iterates $\sigma^{n}_{p,l}$ are
\begin{align}\label{b1a}
X_{\phi_p}\sigma^{n}_{p,l}(x,\theta_p)+
\end{align}
\begin{align}\label{b1b}
\sum^{d-1}_{j=0}\sum^{\mu_p}_{k'=1}a^{k'}_{p,l,j}(\underline{v}^{n-1})\partial_{\theta_p}\sigma^{n}_{p,k'}(x,\theta_p)
+\sum^{d-1}_{j=0}\sum^{\mu_p}_{k=1}\sum^{\mu_p}_{k'=1}b^{k,k'}_{p,l,j}\sigma^{n-1}_{p,k}(x,\theta_p)\partial_{\theta_p}\sigma^{n}_{p,k'}(x,\theta_p)+
\end{align}
\begin{align}\label{b1c}
\sum^{d-1}_{j=0}\sum^{\mu_p}_{k=1}\sum^{\mu_p}_{k'=1}c^{k,k'}_{p,l,j}J^{k,k',n}_{p,n_q,n_r}(x,\theta_p)+
\sum^{d-1}_{j=0}\sum^{\mu_p}_{k=1}\sum^{\mu_p}_{k'=1}d^{k,k'}_{p,l,j}J^{k,k',n}_{p,n_r,n_q}(x,\theta_p)
\end{align}
\begin{align}\label{b1d}
=\sum^{\mu_p}_{k=1}e^k_{p,l}\sigma^{n-1}_{p,k}(x,\theta_p),\quad\textrm{for }p\in \mathcal{P}\cup\mathcal{N},\;l\in \{1,\ldots,\mu_p\},
\end{align}
where we define
\begin{align}
J^{k,k',n}_{p,n_q,n_r}(x,\theta_p)=\frac{1}{2\pi}\int^{2\pi}_0(\sigma^{n-1}_{q,k})_{n_q}\left(x,\frac{n_p}{n_q}\theta_p-\frac{n_r}{n_q}\theta_r\right)\partial_{\theta_r}\sigma^n_{r,k'}(x,\theta_r)d\theta_r,
\end{align}
and similar for $J^{k,k',n}_{p,n_r,n_q}$. Imposed on the iterates $\sigma^n_{p,l}$ are boundary conditions identical to those for the profiles $\sigma_{p,l}$, i.e. \eqref{c10} with $\sigma^n_{m,k}$ in place of each $\sigma_{m,k}$, and initial conditions
\begin{align}
\sigma^{n}_{p,l}(x)=0 \textrm{ for } t \leq 0.
\end{align}
Similar to the elliptic profiles, the elliptic iterates $\sigma^n_{p,l}$ will only approximately solve \eqref{b1a}-\eqref{b1d} in the sense that these equations will hold at $x_d=0$. In fact, the elliptic iterates $\sigma^n_{p,l}$ are obtained with the same kind of construction used for the elliptic $\sigma_{p,l}$. In addition to the elliptic profiles, their corresponding iterates are constructed in Proposition \ref{r2}.

For now, let us fix $p\in\{1,\ldots,M\},\,l\in\{1,\ldots,\mu_p\}$ and an integer $n$ and take the task of finding approximate solutions $\sigma_{p,l}$ and $\sigma^n_{p,l}$ of \eqref{ellipIntEqn} and the system \eqref{b1a}-\eqref{b1d}, respectively. First we rearrange \eqref{ellipIntEqn}, isolating the vector field applied to $\sigma_{p,l}$ in the left hand side, and denoting the resulting right hand side by $f_{p,l}$, getting something of the form
\begin{align}\label{b2exact}
X_{\phi_p}\sigma_{p,l}(x,\theta_p)=f_{p,l}(x,\theta_p).
\end{align}
Similarly, we isolate $X_{\phi_p}\sigma^n_{p,l}$ in the left hand side of \eqref{b1a}-\eqref{b1d}, and define $f^n_{p,l}(x,\theta_p)$ to be what remains on the right hand side, so \eqref{ellipIntEqn} becomes
\begin{align}\label{b2}
X_{\phi_p}\sigma^n_{p,l}(x,\theta_p)=f^n_{p,l}(x,\theta_p).
\end{align}
Now we isolate the quantities this prescribes for the traces of the $\sigma_{p,l}$ and the $\sigma^n_{p,l}$ at the boundary $x_d=0$.
\begin{definition}\label{traceDefns}
(i)\textup{ The coefficient of $\partial_{x_d}\sigma_{p,l}$ in the left hand side of \eqref{b2exact} is one, so we may isolate it in this expression. We will require this equation to hold at $x_d=0$, and define $b_{p,l}$ so that doing so is represented by the condition
\begin{align}\label{rename6exact}
\partial_{x_d}\sigma_{p,l}|_{x_d=0}=b_{p,l},
\end{align}
noting that the right hand side depends only on the boundary data $\sigma_{m,k}|_{x_d=0}$, $m=1,\ldots,M$, $k=1,\ldots,\mu_m$.
}

(ii)\textup{ We define $a_{p,l}$ such that the boundary condition on $\sigma_{p,l}$ in \eqref{c10} is equivalent to
\begin{align}\label{rename5exact}
\sigma_{p,l}|_{x_d=0}=a_{p,l},
\end{align}
noting that the right hand side has already been determined by our solution of the hyperbolic profiles. This thus determines the right hand side of \eqref{rename6exact}.
}

(iii)\textup{ We isolate $\partial_{x_d}\sigma^{n}_{p,l}$ on the left hand side of \eqref{b1a}-\eqref{b1d}, and define $b^n_{p,l}$ such that requiring this to hold at $x_d=0$ is equivalent to
\begin{align}\label{rename6}
\partial_{x_d}\sigma^n_{p,l}|_{x_d=0}=b^n_{p,l}.
\end{align}
}
(iv)\textup{ Consider
\begin{align}\label{rename4}
(\sigma^{n,\pm}_{m,k}(x',0,\theta_0);m\in \mathcal{I}\cup\mathcal{P}\cup\mathcal{N},k\in\{1,\dots,\mu_m\})
=\mathbb{M}^\pm(G^{*,\pm},\sigma^{n,\pm}_{m,k}(x',0,\theta_0);m\in \mathcal{O},k\in\{1,\dots,\mu_m\}),
\end{align}
for $\bM^\pm$ as defined in the proof of Proposition \ref{rename3}. We define $a^n_{p,l}$ such that the condition on $\sigma^n_{p,l}$ in \eqref{rename4} is equivalent to
\begin{align}\label{rename5}
\sigma^n_{p,l}|_{x_d=0}=a^n_{p,l}.
\end{align}
}
\end{definition}

\begin{lem}\label{rename7}
Suppose $G\in H^{s+1}_{T_0}$. Let $\cV^0_h\in \bH^{s+1}_{T_0}$ be the solution of the hyperbolic sub-system constructed in Proposition \ref{cgwprops}. For $a^n_{p,l}$, $b^n_{p,l}$, $a_{p,l}$, and $b_{p,l}$ as defined in Definition \ref{traceDefns}, we have $a^n_{p,l}\in H^{s+1}_{T_0}(x',\theta_p)$, $b^n_{p,l}\in H^{s}_{T_0}(x',\theta_p)$, and
\begin{align}\label{r4}
\lim_{n\rightarrow\infty}a^n_{p,l}=a_{p,l}\textrm{ in }H^{s+1}_{T_0}(x',\theta_p),
\end{align}
\begin{align}\label{r5}
\lim_{n\rightarrow\infty}b^n_{p,l}=b_{p,l}\textrm{ in }H^s_{T_0}(x',\theta_p).
\end{align}
\end{lem}

\begin{proof}
Checking $a^n_{p,l}\in H^{s+1}_{T_0}(x',\theta_p)$ is straightforward, since $G(x',\theta_p)$, $\sigma^{n}_{m,k}(x',0,\theta_p)$, $m\in\mathcal{O}$, are in $H^{s+1}_{T_0}(x',\theta_p)$.\\
Now we show $b^n_{p,l}\in H^{s}_{T_0}(x',\theta_p)$. Observe $b^n_{p,l}$ consists of the terms evaluated at $x_d=0$ \eqref{b1b}, \eqref{b1c}, \eqref{b1d}, and
\begin{align}\label{b1e}
(\partial_{x_d}-X_{\phi_p})\sigma^{n}_{p,l}(x',0,\theta_p).
\end{align}
Since the $\underline{v}^n_{p,k}$, $\sigma^{n}_{p,k}(x',0,\theta_p)$, $p\in\mathcal{I}\cup\mathcal{O}$, are in $H^{s+1}_{T_0}(x',\theta_p)$, and $H^{s}_{T_0}(x',\theta_p)$ is a Banach algebra, it follows that \eqref{b1e} and both \eqref{b1b} and \eqref{b1d} at $x_d=0$ are in $H^{s}_{T_0}(x',\theta_p)$. It remains to show this for \eqref{b1c} at $x_d=0$. This follows from Proposition \ref{rename2} with $x'$ in place of $x$. Taking into account that, for $p\in\mathcal{I}\cup\mathcal{O}$,
\begin{align}\label{r6}
\lim_{n\rightarrow\infty}\sigma^{n}_{p,k}|_{x_d=0}=\sigma_{p,k}|_{x_d=0}\textrm{ in }H^{s+1}_{T_0}(x',\theta_p),
\end{align}
\begin{align}\label{r7}
\lim_{n\rightarrow\infty}\partial_{x_d}\sigma^{n}_{p,k}|_{x_d=0}=\partial_{x_d}\sigma_{p,k}|_{x_d=0}\textrm{ in }H^s_{T_0}(x',\theta_p),
\end{align}
the proof that \eqref{r4} and \eqref{r5} hold is similar.
\end{proof}

\begin{rem}
\textup{
To simplify obtaining our approximate solutions of \eqref{b1a}-\eqref{b1d}, it will be useful to extend functions $a^n_{p,l},a_{p,l}\in H^{s+1}_{T_0}=H^{s+1}((-\infty,T_0)\times\bR^{d-1}\times\bT)$, $b^n_{p,l},b_{p,l}\in H^s_{T_0}=H^s((-\infty,T_0)\times\bR^{d-1}\times\bT)$ on the half-space to elements of $H^{s+1}(\bR^d\times\bT)$ and $H^s(\bR^d\times\bT)$, respectively.
}
\end{rem}
\begin{lem}\label{new52}
For $s\geq 0$ there is a continuous extension map
\begin{align}
E:H^s((-\infty,T_0)\times\bR^{d-1}\times\bT)\rightarrow H^s(\bR^d\times\bT).
\end{align}
\end{lem}
\begin{proof}
It is shown in 4.4 of \cite{taylor} that there is a continuous extension map from $H^s(\bR^d_+)$ to $H^s(\bR^d)$.
\end{proof}
From now on, in place of $a^n_{p,l},a_{p,l},b^n_{p,l},b_{p,l}$ we refer to their extensions to the respective spaces mentioned above unless we explicitly state otherwise.\\The following lemma is a standard kind of result in the theory of hyperbolic initial/boundary value problems.
\begin{lem}\label{rename9}
(i) For $a\in H^{s+1}(\bR^d\times\bT)$ and $b\in H^s(\bR^d\times\bT)$, there exists unique $\varsigma\in H^{s+1}_{D}=H^{s+1}(\bR^d\times[0,D]\times\bT)$ solving the wave equation
\begin{align}\label{rename12}
\partial^2_{x_d}\varsigma-\Delta_{x',\theta_0}\varsigma=0,
\end{align}
and initial-$x_d$ conditions
\begin{eqnarray}\label{rename13}
\varsigma|_{x_d=0}&=& a,\\
\partial_{x_d}\varsigma|_{x_d=0}&=& b.
\end{eqnarray}
(ii) Furthermore, $\varsigma$ satisfies
\begin{align}\label{rename8}
|\varsigma|_{H^{s+1}_D}\leq C\left(|a|_{H^{s+1}}+|b|_{H^s}\right),
\end{align}
for some constant $C$ independent of the choice of initial data $\{a,b\}$.
\end{lem}
Justification of Lemma \ref{rename9} can be found in Chapter 6 (see 6.18 and Remark 6.21) of \cite{cp}.
\begin{prop}\label{r2}
For $a^n_{p,l}$, $b^n_{p,l}$, $a_{p,l}$, and $b_{p,l}$ as defined in Definition \ref{traceDefns}, where $p\in\cP\cup\cN$ there exist $\sigma^n_{p,l}(x,\theta_p)$, $\sigma_{p,l}(x,\theta_p)\in H^{s+1}_{T_0}(x,\theta_p)$ with compact $x_d$-support in $[0,D]$ satisfying
\begin{align}
\sigma^n_{p,l}=\sigma_{p,l} =0,\textrm{ for }t\leq 0,
\end{align}
boundary conditions
\begin{eqnarray}\label{rename11}
\sigma^n_{p,l}|_{x_d=0}&=&a^n_{p,l},\quad \sigma_{p,l}|_{x_d=0}=a_{p,l},\\
\partial_{x_d}\sigma^n_{p,l}|_{x_d=0}&=&b^n_{p,l},\quad \partial_{x_d}\sigma_{p'l}|_{x_d=0}=b_{p,l}, \nonumber
\end{eqnarray}
and
\begin{align}\label{rename14}
\lim_{n\rightarrow\infty}\sigma^n_{p,l}=\sigma_{p,l}\textrm{ in }H^{s+1}_{T_0}(x,\theta_p).
\end{align}
\end{prop}
\begin{proof}
We apply Lemma \ref{rename9} to initial data $\{a_{p,l},b_{p,l}+\partial_t a_{p,l}\}$ to obtain the corresponding solution of \eqref{rename12}, denoted $\varsigma_{p,l}\in H^{s+1}_D$, and subsequently to each $\{a^n_{p,l},b^n_{p,l}+\partial_t a^n_{p,l}\}$, obtaining solutions $\varsigma^n_{p,l}\in H^{s+1}_D$. Now we fix a smooth cutoff function $\chi(x_d)$ supported in $[0,D)$, and define
\begin{eqnarray}\label{r3}
\sigma_{p,l}(t,x'',\theta_p) &=& \chi(x_d) \varsigma_{p,l}(t-x_d,x'',\theta_p),\\
\sigma^n_{p,l}(t,x'',\theta_p) &=& \chi(x_d) \varsigma^n_{p,l}(t-x_d,x'',\theta_p).
\end{eqnarray}
It is easy to check that then $\sigma^n_{p,l}$, $\sigma_{p,l}$ also belong to $H^{s+1}_D$. Since $\varsigma_{p,l}|_{x_d=0}$, $\varsigma^n_{p,l}|_{x_d=0}$ are supported in $t>0$, by finite speed of propagation for solutions to the wave equation \eqref{rename12}, for any fixed $r>0$, the $(t,y,\theta_p)$-supports of $\varsigma_{p,l}|_{x_d=r}$, $\varsigma^n_{p,l}|_{x_d=r}$ are contained in the union of balls of radius $r$ about points in $\{(t,y,\theta_p):t>0\}$. Thus, the support is contained in $\{(t,y,\theta_p):t>-r\}$. It follows that $\sigma_{p,l}(t,y,x_d,\theta_p)$, $\sigma^n_{p,l}(t,y,x_d,\theta_p)$ are zero for all $t\leq 0$, since the supports of $\sigma_{p,l}$, $\sigma^n_{p,l}$ consist only of points satisfying $t-x_d>-x_d$, i.e. $t>0$. It is easy to check from \eqref{r3} that the $\sigma_{p,l}$, $\sigma^n_{p,l}$ have the desired traces at $x_d=0$, satisfying \eqref{rename11}. To establish \eqref{rename14}, first we note we may regard the $\sigma_{p,l}$, $\sigma^n_{p,l}$ as elements of $H^{s+1}_{T_0}(x,\theta_p)$ by defining them to be equal to zero for $x_d>D$ and restricting to $t<T_0$. One easily verifies the bound
\begin{align}
|\sigma^n_{p,l}-\sigma_{p,l}|_{H^{s+1}_{T_0}}\leq |\sigma^n_{p,l}-\sigma_{p,l}|_{H^{s+1}_D}\leq C_1 |\varsigma^n_{p,l}-\varsigma_{p,l}|_{H^{s+1}_D}.
\end{align}
Applying Lemma \ref{rename9} to the solutions $\varsigma^n_{p,l}-\varsigma_{p,l}$, from the above and the estimate \eqref{rename8} we get
\begin{align}
|\sigma^n_{p,l}-\sigma_{p,l}|_{H^{s+1}_{T_0}}\leq C_2\left(|a^n_{p,l}-a_{p,l}|_{H^{s+1}}+|b^n_{p,l}-b_{p,l}|_{H^s}\right).
\end{align}
Now we conclude \eqref{rename14} from Lemma \ref{rename7} and the continuity of the extension map from Lemma \ref{new52}.
\end{proof}

Now that we have obtained the elliptic profiles, we have determined the ansatz $\cV^0$.

\begin{definition}\label{ansatzDefn}
Let $G\in H^{s+1}_{T_0}$ and let $\cV^0_h\in \bH^{s+1}_{T_0}$ be the solution of the hyperbolic sub-system constructed in Proposition \ref{cgwprops}, and let the $\cV^{0,n}_h\in \bH^{s+1}_{T_0}$ be the corresponding iterates. Additionally, let $\sigma_{p,l}$, $\sigma^n_{p,l}\in H^{s+1}_{T_0}$, where $p\in\cP\cup\cN$, be the elliptic profiles and iterates obtained in Proposition \ref{r2}.

(i) We define the elliptic part of our ansatz
\begin{align}\label{ellipForm}
\cV^0_e(x,\theta_1,\ldots,\theta_M):=\sum^M_{m\in\cP\cup\cN}\sum^{\mu_m}_{k=1}\sigma_{m,k}(x,\theta_m)r_{m,k},
\end{align}
and the corresponding $n$th iterate
\begin{align}
\cV^{0,n}_e(x,\theta_1,\ldots,\theta_M):=\sum^M_{m\in\cP\cup\cN}\sum^{\mu_m}_{k=1}\sigma^n_{m,k}(x,\theta_m)r_{m,k}.
\end{align}
Our ansatz is defined to be
\begin{align}
\cV^0:=\cV^0_h +\cV^0_e,
\end{align}
and we also define $\cV^{0,n}:=\cV^{0,n}_h +\cV^{0,n}_e$.

(ii) We isolate the error in the iterated interior profile equations, defining $R^n(x,\theta)$ by
\begin{align}\label{d1}
R^n:=\mathbb{E}(\tilde{L}(\partial_x)\mathcal{V}^{0,n}+\mathcal{M}(\mathcal{V}^{0,n-1})\partial_\theta\mathcal{V}^{0,n}-F(0)\mathcal{V}^{0,n-1}).
\end{align}
\end{definition}

\begin{prop}
Suppose the hypotheses of Definition \ref{ansatzDefn} are satisfied. Then $\cV^0\in \bH^{s+1}_{T_0}$, and 
\begin{align}
\lim_{n\to\infty}\cV^{0,n}=\cV^0\textrm{ in }\bH^s_{T_0}.
\end{align}
\end{prop}
\begin{proof}
The claims follow directly from Proposition \ref{cgwprops} and Proposition \ref{r2}.
\end{proof}

\begin{rem}\label{new5}
\textup{Recalling \eqref{hypConditionsIt}, note that
\begin{align}
\mathbb{E}_h(\tilde{L}(\partial_x)\mathcal{V}^{0,n}+\mathcal{M}(\mathcal{V}^{0,n-1})\partial_\theta\mathcal{V}^{0,n}-F(0)\mathcal{V}^{0,n-1})=0,
\end{align}
and so $R^n$ of \eqref{d1} is purely elliptic in the sense that $\bE_e R^n= R^n$.  Moreover,
\begin{align}\label{new6}
R^n=\sum_{p\in\mathcal{P}\cup\mathcal{N}}\sum^{\mu_p}_{l=1}
(X_{\phi_p}\sigma^{n}_{p,l}-f^{n}_{p,l})r_{p,l},
\end{align}
which is zero at the boundary $x_d=0$, since we have satisfied \eqref{b2exact} for $x_d=0$.
Summarizing the results of Proposition \ref{cgwprops} and Proposition \ref{r2}, we conclude
\begin{align}\label{fullConditionsIt}
\begin{split}
&a)\;\bE\cV^{0,n}=\cV^{0,n}\\
&b)\;\bE\left(\tilde{L}(\partial_x)\cV^{0,n}+\cM(\cV^{0,n-1})\partial_{\theta}\cV^{0,n}\right)=\bE(F(0)\cV^{0,n-1})+R^n\text{ in }x_d\geq 0\\
&c)\;B(0)\cV^{0,n}(x',0,\theta_0,\dots,\theta_0)=G(x',\theta_0)\\
&d)\;\cV^{0,n}=0\text{ in }t<0.
\end{split}
\end{align}
The sense in which an error such as $R^n$ is small is clarified by Proposition \ref{new3}.}
\end{rem}

\begin{rem}
\textup{Recall the discussion in Remark \ref{modeProps}(ii) which gives the bijection between the eigenvalues $\uomega_m$ with $m\in\cP$ and $\uomega_{m'}=\overline{\uomega}_m$, $m'\in\cN$ and between the corresponding eigenvectors. A careful look at the profile equations shows that the equations for the elliptic profiles come in conjugate pairs. That is, the equation for $\sigma_{m,k}$, some $m\in\cP$, is the conjugate of the equation for $\sigma_{m',k'}$ with the corresponding $m'\in\cN$ and $k'$. As a result the solutions we have obtained, the elliptic profiles, also come in conjugate pairs, satisfying $\sigma_{m',k'}=\overline{\sigma_{m,k}}$. Meanwhile, the hyperbolic part of the solution $\cV^0_h$ is real. It follows from these observations and the definition of our ansatz that when we plug in $\theta=\theta(\theta_0,\xi_d)$, getting
\begin{align}\label{blank}
\cU^0(x,\theta_0,\xi_d):=\cV^0(x,\theta_0+\uomega_1 \xi_d,\ldots, \theta_0+\uomega_M \xi_d),
\end{align}
we have a function $\cU^0(x,\theta_0,\xi_d)$ which is in fact real on $\overline{\R}^{d+1}_+\times\bT\times\overline{\R}_+$.}
\end{rem}
\subsection{The expansions for the approximate solution }


For this study, we define the spaces $E^s_T$ and $\cE^s_T$ as in \cite{cgw}.
\begin{align}\label{ESpace}
E^s_T=C(x_d,H^s_T(x',\theta_0))\cap L^2(x_d,H^{s+1}_T(x',\theta_0)),
\end{align}
where by $C(x_d,H^s_T(x',\theta_0))$ we actually refer to functions in $C(x_d,H^s_T(x',\theta_0))$ with $x_d$-support in $[0,D]$ for some large enough $D$, and for $C(x_d,H^s_T(x',\theta_0))$ we use the $L^\infty(x_d,H^s_T(x',\theta_0))$ norm where the supremum is taken over $x_d\geq 0$. These spaces are algebras and are contained in $\L^\infty$ for $s>\frac{d+1}{2}$. Theorem 7.1 of \cite{williams2}, as discussed in Section \ref{testlabel}, regarding existence of solutions of the singular system, tells us that for $s$ large enough, one has existence of solutions to \eqref{new61} in the space $E^s_T$ on a time interval $[0,T]$ independent of $\eps\in (0,\eps_0]$. For these reasons, the space $E^s_T$ and related estimates are key to our analysis, in particular for our main theorem (Theorem \ref{e1}) which relies on Proposition \ref{e9}, also proved in \cite{williams2}. Additionally we use the spaces
\begin{align}\label{d21uz}
\cE^s_T=\{\cU(x,\theta_0,\xi_d):\sup_{\xi_d\geq 0}|\cU(\cdot,\cdot,\xi_d)|_{E^s_T} < \infty\},
\end{align}
which also play a role in Theorem \ref{e1}. The proof of Theorem \ref{e1} uses estimates regarding functions such as $\cV$ in $H^{s+1}_T$ and the corresponding $\cU$ in $\cE^s_T$ which results from the substitution $\theta=\theta(\theta_0,\xi_d)$. While this moves us out of the space of periodic profiles  $H^{s+1}_T(x,\theta)$, we get functions in $\cE^s_T(x,\theta_0,\xi_d)$ which we can approximate with finite trigonometric polynomials (truncated expansions\footnote{These expansions are given by Proposition \ref{propA}.}) in $\theta=\theta(\theta_0,\xi_d)$.
\begin{definition}
For $k=1,\,2$ we define
\begin{align}
\mathcal{E}^{s;k}_T:=\{\mathcal{U}(x,\theta_0,\xi_d)=\mathcal{V}(x,\theta)|_{\theta=\theta(\theta_0,\xi_d)}:\mathcal{V}\in H^{s+1;k}_T\},
\end{align}
with the norm $|\cdot|_{\mathcal{E}^s_T}$. (Note: The subscript $T$ has the same indication as it did for the $H^{s+1;k}_T$ used in \textup{\cite{cgw}}, where the subscript $T$ is at first ignored.)
\end{definition}
It is verified in Proposition \ref{propA} that elements of $\mathcal{E}^{s;2}_T$ are bounded in the $\mathcal{E}^s_T$ norm.

\textbf{Convergence of expansions in $\cE^s_T$. }\label{profileeqns}

The following notation gives us a way to sort the spectra $\alpha\in Z^{M;2}$ of elements $\mathcal{V}$ in $H^{s+1;2}_T$, which will aid in showing $\mathcal{U}(x,\theta_0,\xi_d)=\mathcal{V}(x,\theta)|_{\theta=\theta(\theta_0,\xi_d)}$ has an expansion converging in $\mathcal{E}^s_T$, with Proposition \ref{propA}.
\begin{definition}\label{defA}
Let $\{z_k\}^\infty_{k=1}$ enumerate the set $\{\alpha\cdot\uomega:\alpha\in Z^{M;2}\}$. For each $j,k$, we define
\begin{align}
C_{j,k}:=\{\alpha\in Z^{M;2}:\sum^M_{i=1}\alpha_i=j,\alpha\cdot\uomega=z_k\}
\end{align}
and pick out one element $\alpha_{(j,k)}\in C_{j,k}$.
\end{definition}
\begin{rem}
\emph{
A nice consequence of the fact that we work with $Z^{M;2}$ instead of general $Z^{M;k}$ is that we can easily show the $C_{j,k}$ are finite\footnote{We work with $Z^{M;2}$ due to the fact that we only have to consider quadratic interactions. Further discussion on interactions and resonances can be found in Sections \ref{peransatzwithcor} and \ref{largesystem}.}. To see this, first fix $j\in\Z$, $k\in\{1,2,\ldots\}$, and $p$, $q\in\{1,\ldots,M\}$ and take arbitrary $\alpha\in C_{j,k}$ such that all but the $p$th and $q$th components are zero (either of the $p$th and $q$th components may be zero, as well.) We claim this is the only such element of $C_{j,k}$. This is because $\uomega_p\neq\uomega_q$ implies
\begin{align}
\left( \begin{array}{ccc}
1        & 1 \\
\uomega_p & \uomega_q \end{array} \right)
\left( \begin{array}{ccc}
\alpha_p \\
\alpha_q \end{array} \right)=
\left( \begin{array}{ccc}
j \\
z_k \end{array} \right)
\end{align}
has a unique solution $(\alpha_p,\alpha_q)$. Thus, the number of such $\alpha$ in $C_{j,k}$ is bounded by the number of pairs of components, so $|C_{j,k}|\leq M(M-1)/2$.
}
\end{rem}

A function $\mathcal{V}(x,\theta)\in H^{s+1;2}_T(x,\theta)$ of $(x,\theta)\in\bR^{d+1}\times\bold{C}^M$ has a series
\begin{align}\label{a3}
\mathcal{V}(x,\theta)=\sum_{\alpha\in Z^{M;2}}V_\alpha(x)e^{i\alpha\cdot\theta}
\end{align}
and, for fixed $x_d$, squared $H^s_T(x',\theta)$ norm
\begin{align}\label{a4}
|\mathcal{V}(x,\theta)|^2_{H^s_T(x',\theta)}=\sum_{\alpha\in Z^{M;2}}\sum_{|\beta |\leq s}|\partial^\beta_{x'}V_\alpha(x)|^2_{L^2(x')}(1+|\alpha |)^{2(s-|\beta |)}.
\end{align}
From Sobolev embedding and the fact that
\begin{align}
H^{s+1}_T(x,\theta)\subset L^2(x_d,H^{s+1}_T(x',\theta))\cap H^1(x_d,H^s_T(x',\theta)),
\end{align}
we find $\mathcal{V}(x,\theta)\in L^2(x_d,H^{s+1}_T(x',\theta))\cap C(x_d,H^s_T(x',\theta))$, implying the partial sums of the series \eqref{a3} are bounded and converge in $H^s_T(x',\theta)$ uniformly with respect to $x_d\geq 0$. We will prove Proposition \ref{propA} by using these facts with the following lemma, which shows for finite truncations of expansions \eqref{a3} that the norm $|\cdot|_{H^s_T(x',\theta)}$ dominates $\left|\cdot|_{\theta=\theta(\theta_0,\xi_d)}\right|_{H^s_T(x',\theta_0)}$ independent of $(x_d,\xi_d)$.
\begin{lem}\label{1stlem}
Suppose $\mathcal{V}\in H^{s+1;2}_T(x,\theta)$ with series given by \eqref{a3}. For $\theta(\theta_0,\xi_d)$ as in \eqref{new15} and integers $M_1 \leq M_2$, $0<N_1 \leq N_2$, we have the following inequality:
\begin{align}\label{a9}
|\sum^{M_2}_{j=M_1}\sum^{N_2}_{k=N_1} \sum_{\alpha\in C_{j,k}} V_\alpha(x) e^{i\alpha\cdot\theta(\theta_0,\xi_d)}|^2_{H^s_T(x',\theta_0)}
\leq
|\sum^{M_2}_{j=M_1}\sum^{N_2}_{k=N_1}\sum_{\alpha \in C_{j,k}}V_\alpha(x)e^{i\alpha\cdot\theta}|^2_{H^s_T(x',\theta)}.
\end{align}
\end{lem}
\begin{proof}
We estimate
\begin{eqnarray}
|\sum^{M_2}_{j=M_1}\sum^{N_2}_{k=N_1} \sum_{\alpha\in C_{j,k}} V_\alpha(x) e^{i\alpha\cdot\theta(\theta_0,\xi_d)}|^2_{H^s_T(x',\theta_0)} 
&=& |\sum^{M_2}_{j=M_1}\sum^{N_2}_{k=N_1}(\sum_{\alpha \in C_{j,k}}V_\alpha(x)e^{i\alpha\cdot\omega\xi_d})e^{ij\theta_0}|^2_{H^s_T(x',\theta_0)}, \nonumber \\
&=& \sum^{M_2}_{j=M_1}\sum_{|\beta |\leq s}|(\sum^{N_2}_{k=N_1}\sum_{\alpha \in C_{j,k}}\partial^\beta_{x'}V_\alpha(x)e^{i\alpha\cdot\omega\xi_d})|^2_{L^2(x')}(1+|j|)^{2(s-|\beta |)}, \label{a6} \\
&\leq & \sum^{M_2}_{j=M_1}\sum^{N_2}_{k=N_1} \sum_{\alpha \in C_{j,k}}\sum_{|\beta |\leq s}|\partial^\beta_{x'}V_\alpha(x)e^{i\alpha\cdot\omega\xi_d}|^2_{L^2(x')}(1+|j|)^{2(s-|\beta |)}. \label{a7}
\end{eqnarray}
For \eqref{a6}, we used the formula which that in \eqref{a4} generalizes. For $\alpha\in C_{j,k}$, we have $|j|=|\sum_i\alpha_i|\leq|\alpha|$ and $\textrm{Im}(\alpha\cdot\omega)\geq 0$, so for all $\xi_d\geq 0$, the sum in \eqref{a7} is bounded by
\begin{align}\label{a8}
|\sum^{M_2}_{j=M_1}\sum^{N_2}_{k=N_1}\sum_{\alpha \in C_{j,k}}V_\alpha(x)e^{i\alpha\cdot\theta}|^2_{H^s_T(x',\theta)}
=\sum^{M_2}_{j=M_1}\sum^{N_2}_{k=N_1}\sum_{\alpha \in C_{j,k}}\sum_{|\beta |\leq s}|\partial^\beta_{x'}V_\alpha(x)|^2_{L^2(x')}(1+|\alpha|)^{2(s-|\beta |)},
\end{align}
where the equality \eqref{a8} follows from the formula in \eqref{a4}.
\end{proof}

With the estimate of Lemma \ref{1stlem}, we are ready to prove that elements of $H^{s+1;2}_T$ yield elements with expansions converging in $\mathcal{E}^s_T$ upon the substitution $\theta=\theta(\theta_0,\xi_d)$.

\begin{prop}\label{propA}
Let $\mathcal{V}\in H^{s+1;2}_T(x,\theta)$ with expansion \eqref{a3} and set $\mathcal{U}=\mathcal{V}|_{\theta=\theta(\theta_0,\xi_d)}$ for $\theta(\theta_0,\xi_d)$ as in \eqref{new15}. Then we have $\mathcal{U}\in\mathcal{E}^s_T$:
\begin{align}\label{a2}
|\mathcal{U}|_{\mathcal{E}^s_T}\leq C |\mathcal{V}|_{H^{s+1}_T},
\end{align}
and we have convergence in $\mathcal{E}^s_T$ to $\mathcal{U}$ of finite partial sums, independent of arrangement, of the series
\begin{align}\label{a18}
\mathcal{U}(x,\theta_0,\xi_d)=\sum_{\alpha\in Z^{M;2}}V_\alpha(x)e^{i\alpha\cdot\theta(\theta_0,\xi_d)}.
\end{align}
\end{prop}

\begin{proof}[Proof of Proposition \ref{propA}]
We explicitly show the convergence of the finite partial sums
\begin{align}\label{ab1}
\sum_{ \begin{subarray}{c} -n \leq j \leq n\\ 1 \leq k \leq n\end{subarray}}\sum_{\alpha\in C_{j,k}} V_\alpha(x) e^{i\alpha\cdot\theta(\theta_0,\xi_d)}.
\end{align}
For integers $n_1$, $n_2$ with $n_1\leq n_2$, an application of Lemma \ref{1stlem} gives
\begin{align}\label{ab2}
|\sum_{ \begin{subarray}{} n_1 \leq |j| \leq n_2 \\ n_1 \leq k \leq n_2 \end{subarray}}\sum_{\alpha\in C_{j,k}} V_\alpha(x) e^{i\alpha\cdot\theta(\theta_0,\xi_d)}|^2_{H^s_T(x',\theta_0)}
\leq
|\sum_{ \begin{subarray}{} n_1 \leq |j| \leq n_2 \\ n_1 \leq k \leq n_2 \end{subarray}}\sum_{\alpha \in C_{j,k}}V_\alpha(x)e^{i\alpha\cdot\theta}|^2_{H^s_T(x',\theta)}.
\end{align}
It follows from \eqref{ab2} that since the sequence of the partial sums
\begin{align}
\sum_{ \begin{subarray}{c} -n \leq j \leq n\\ 1 \leq k \leq n\end{subarray}}\sum_{\alpha\in C_{j,k}} V_\alpha(x) e^{i\alpha\cdot\theta}
\end{align}
is Cauchy in $H^s_T(x',\theta)$ uniformly with respect to $(x_d,\xi_d)$, so is the sequence \eqref{ab1} in $H^s_T(x',\theta_0)$, and thus we have convergence. We similarly get convergence of the \eqref{ab1} in $L^2(x_d,H^{s+1}_T(x',\theta_0))$ after integrating the inequality \eqref{ab2} (with $s+1$ in place of $s$) with respect to $x_d$ and noting the right hand side is independent of $\xi_d$. Hence, the \eqref{ab1} converge in $\mathcal{E}^s_T$. It is not hard to show the limit is in fact $\mathcal{V}|_{\theta=\theta(\theta_0,\xi_d)}$, and the estimate \eqref{a2} easily follows. The same proof works for arbitrarily ordered sums, where one instead considers finite $\cB_n \nearrow Z^{M;2}$ and similarly gets that the sequence of the
\begin{align}\label{new33}
\sum_{\alpha\in \cB_n} V_\alpha(x) e^{i\alpha\cdot\theta(\theta_0,\xi_d)}
\end{align}
is Cauchy in the desired space because the sequence of the $\sum_{\alpha\in \cB_n} V_\alpha(x) e^{i\alpha\cdot\theta}$ is Cauchy.
\end{proof}

\subsection{Error analysis}\label{erroranalysis}
The following proposition is used to make precise the notion that $\sigma_{p,l}(x,\theta_p)$ and $\sigma^n_{p,l}(x,\theta_p)$ are approximate solutions of \eqref{ellipIntEqn} and \eqref{b1a}-\eqref{b1d}, respectively.
\begin{prop}\label{new3}
Let $R(x,\theta)\in H^{s;1}_T(x,\theta)$ have the property that it is polarized by the elliptic projector $\bE_e$ defined in \eqref{new14}, i.e. that $\bE_e(R)=R$, and suppose
\begin{align}
R|_{x_d=0}=0.
\end{align}
Then
\begin{align}\label{new13}
\lim_{\epsilon\rightarrow 0}|R(x,\theta_0+\uomega_1\xi_d,\ldots,\theta_0+\uomega_M\xi_d)|_{\xi_d=\frac{x_d}{\eps}}|_{E^{s-1}_T(x,\theta_0)}=0.
\end{align}
\end{prop}
\begin{proof}
It suffices to consider the case $R(x,\theta_p)\in H^s_T(x,\theta_p)$, $\bE_p(R)=R$, for some $p\in\mathcal{P}\cup\mathcal{N}$, and show
\begin{align}
\lim_{\epsilon\rightarrow 0}|R(x,\theta_0+\uomega_p\xi_d)|_{\xi_d=\frac{x_d}{\epsilon}}|_{E^{s-1}_T(x,\theta_0)}=0.
\end{align}
First, we show
\begin{align}\label{new0}
\lim_{\epsilon\rightarrow 0}\sup_{x_d\geq 0}|R(x,\theta_0+\uomega_p\xi_d)|_{\xi_d=\frac{x_d}{\epsilon}}|_{H^{s-1}_T(x',\theta_0)}=0.
\end{align}
Note $R$ has an expansion of the form
\begin{align}
R(x,\theta_p)=\sum_{j\in Z_p\setminus 0} a_j(x)e^{ij\theta_p}.
\end{align}
Thus, noting $R(x,\theta_0)\in H^s_T(x,\theta_0)\subset C(x_d,H^{s-1}_T(x',\theta_0))$, fixing $x_d$, we get the norm
\begin{align}
|R(x,\theta_p)|^2_{H^{s-1}_T(x',\theta_p)}=\sum_{j\in Z_p\setminus 0}\sum_{|\beta |\leq s-1}|\partial^\beta_{x'}a_j(x)|^2_{L^2(x')}(1+|j|)^{2(s-1-|\beta |)}.
\end{align}
We let $\xi_d=\frac{x_d}{\epsilon}$ and find
\begin{align}
R(x,\theta_0+\uomega_p\xi_d)=\sum_{j\in Z_p\setminus 0} ( e^{-\textrm{Im}(j\uomega_p)\xi_d}e^{i\textrm{Re}(j\uomega_p)\xi_d}a_j(x))e^{ij\theta_0}.
\end{align}
Then it follows
\begin{eqnarray}
|R(x,\theta_0+\uomega_p\xi_d)|^2_{H^{s-1}_T(x',\theta_0)}
&=&\sum_{j\in Z_p\setminus 0}\sum_{|\beta |\leq s-1}(e^{-\textrm{Im}(j\uomega_p)\xi_d})^2|\partial^\beta_{x'}a_j(x)|^2_{L^2(x')}(1+|j|)^{2(s-1-|\beta |)}\nonumber\\
&\leq& e^{-2|\textrm{Im}\,\uomega_p|\xi_d}\sum_{j\in Z_p\setminus 0}\sum_{|\beta |\leq s-1}|\partial^\beta_{x'}a_j(x)|^2_{L^2(x')}(1+|j|)^{2(s-1-|\beta |)}\\
&=&  e^{-2|\textrm{Im}\,\uomega_p|\xi_d} |R(x,\theta_0)|^2_{H^{s-1}_T(x',\theta_0)}
\end{eqnarray}
Now we show that, as $\epsilon$ tends to zero,
\begin{align}\label{new1}
\sup_{x_d\in[0,\sqrt{\epsilon}]}|R(x,\theta_0+\uomega_p\frac{x_d}{\epsilon})|_{H^{s-1}_T(x',\theta_0)}\rightarrow 0,
\end{align}
and
\begin{align}\label{new2}
\sup_{x_d\geq\sqrt{\epsilon}}|R(x,\theta_0+\uomega_p\frac{x_d}{\epsilon})|_{H^{s-1}_T(x',\theta_0)}\rightarrow 0.
\end{align}
Set $h(x_d)=R(x',x_d,\theta_0)\in H^{s-1}_T(x',\theta_0)$. To see \eqref{new1}, observe that the term is bounded by
\begin{align}
\sup_{x_d\in[0,\sqrt{\epsilon}]}e^{-2|\textrm{Im}\,\uomega_p|\frac{x_d}{\epsilon}} |h(x_d)|^2_{H^{s-1}_T(x',\theta_0)}
\leq\sup_{x_d\in[0,\sqrt{\epsilon}]}|h(x_d)|^2_{H^{s-1}_T(x',\theta_0)},
\end{align}
which converges to $0$ as $\epsilon$ tends to zero, since $h \in C(x_d,H^{s-1}_T(x',\theta_0))$ with $h(0)=0$. That \eqref{new2} holds follows from the fact that this term is bounded by
\begin{align}
\sup_{x_d\geq\sqrt{\epsilon}}e^{-2|\textrm{Im}\,\uomega_p|\frac{x_d}{\epsilon}} |h(x_d)|^2_{H^{s-1}_T(x',\theta_0)}
\leq \sup_{x_d\geq 0}|h(x_d)|^2_{H^{s-1}_T(x',\theta_0)}e^{-2|\textrm{Im}\,\uomega_p|\frac{1}{\sqrt{\epsilon}}},
\end{align}
which also converges to $0$ as $\epsilon$ tends to zero. Now we check that
\begin{align}
\lim_{\epsilon\rightarrow 0}|R(x,\theta_0+\uomega_p\xi_d)|_{\xi_d=\frac{x_d}{\epsilon}}|_{L^2(x_d,H^s_T(x',\theta_0))}=0.
\end{align}
It is not hard to show that
\begin{align}
|R(x,\theta_0+\uomega_p\frac{x_d}{\epsilon})|^2_{L^2(x_d,H^s_T(x',\theta_0))}
\leq |e^{-|\textrm{Im}\,\uomega_p|\frac{x_d}{\epsilon}}|_{L^2} |R(x,\theta_0)|^2_{L^2(x_d,H^s_T(x',\theta_0))}\leq D \sqrt{\epsilon},
\end{align}
which completes the proof.
\end{proof}

 Now we are ready to show that, as $\epsilon\rightarrow 0$, the approximate solution $u^a_\epsilon$ converges to the exact solution $u_\epsilon$ in $L^\infty$. This is a corollary of the following theorem.
\begin{theo}\label{e1}
For $M_0=2(d+2)+1$ and $s\geq 1+[M_0+\frac{d+1}{2}]$ let $G(x',\theta_0)\in H^{s+1}_T$ have compact support in $x'$ and vanish in $t\leq 0$.  Let $U_\eps(x,\theta_0)\in E^s_{T_0}$ be the exact solution to the singular system for $0<\eps\leq \eps_0$ given by Theorem 7.1 of \cite{williams2}, let $\cV^0\in \bH^{s+1}_{T_0}$ be the profile given by \eqref{new65}, and let $\cU^0\in\cE^s_{T_0}$ be defined by
\begin{align}\label{e2}
\cU^0(x,\theta_0,\xi_d)=\cV^0(x,\theta_0+\uomega_1\xi_d,\dots,\theta_0+\uomega_M\xi_d).
\end{align}
Here $0 < T_0\leq T$ is the minimum of the existence times for the quasilinear problems \eqref{new61} and \eqref{fullConditions}.
Define
\begin{align}\label{e3}
\cU^0_\eps(x,\theta_0):=\cU^0(x,\theta_0,\frac{x_d}{\eps}).
\end{align}
The family $\cU^0_\eps$ is uniformly bounded in $E^s_{T_0}$ for $0<\eps\leq \eps_0$ and satisfies
\begin{align}\label{e4}
|U_\eps-\cU^0_\eps|_{E^{s-1}_{T_0}}\to 0 \text{ as }\eps\to 0.
\end{align}
\end{theo}

Upon evaluating our placeholder $\xi_d=\frac{x_d}{\eps}$ in the argument of $\cU^0\in\cE^s_T$, we get the function $\cU^0_\eps$ in $E^s_T$, the same space containing the solutions of the singular system. Indeed, the $E^s_T$ norm is that of the estimate of Proposition \ref{e9}, key to our proof of Theorem \ref{e1}, and the norm in which we show our approximate solution $\cU^0_\eps$ is close to the exact solution of the singular system.
\begin{lem}\label{e11}
For $m\geq 0$ suppose $\cV(x,\theta)\in \bH^{m+1}_T$,  $\bE\cV=\cV$, and set $\cU(x,\theta_0,\xi_d)=\cV(x,\theta_0+\uomega_1\xi_d,\dots,\theta_0+\uomega_M\xi_d)$.
Then, setting $\cU_{\eps}(x,\theta_0)=\cU(x,\theta_0,\frac{x_d}{\eps})$,
\begin{align}\label{e13}
|\cU_{\eps}|_{E^m_T}\leq |\cU|_{\cE^m_T}.
\end{align}
\end{lem}

\begin{proof}
The proof is almost identical to the argument used in \cite{cgw} to prove Lemma 2.25 (b) with Lemma 2.7.
\end{proof}

\begin{proof}[Proof of Theorem \ref{e1}]
It suffices to prove boundedness of the family $\cU^0_{\eps}$ in $E^s_{T_0}$ along with the following three statements:
\begin{align}\label{e8}
\begin{split}
&(a)\;\lim_{n\to\infty}U^n_\eps= U_\eps\text{ in }   E^{s-1}_{T_0} \text{ uniformly  with respect to }\eps\in (0,\eps_0]\\
&(b) \;\lim_{n\to\infty}\cU^{0,n}_{\eps}= \cU^0_{\eps}\text{ in }   E^{s-1}_{T_0} \text{  uniformly  with respect to }\eps\in (0,\eps_0]\\
&(c) \;\text{For each }n\;\; \lim_{\eps\to 0}|U^n_\eps-\cU^{0,n}_{\eps}|_{E^{s-1}_{T_0}}=0.
\end{split}
\end{align}

The uniform boundedness of $U_\eps$ and the $U^n_\eps$ in $E^s_{T_0}$ and that \eqref{e8}(a) holds are proved in \cite{williams2}, Theorem 7.1, with the use of the iteration scheme \eqref{e5} and the following linear estimate:

\begin{prop}[\cite{williams2}, Cor. 7.2]\label{e9}
Let $s\geq [M_0+\frac{d+1}{2}]$ and consider the  problem \eqref{e5}, where $G\in H^{s+1}_T$ has compact support and vanishes in $t\leq 0$, and where the right side of \eqref{e5}(a) is replaced by $\cF\in \E^s_T$ with $\mathrm{supp}\; \cF\subset \{t\geq 0, 0\leq x_d\leq E\}$.
Suppose $U^n_\eps\in E^s_T$ has compact $x-$support and that for some $K>0$, $\eps_1>0$  we have
\begin{align}\label{e9a}
|U^n_\eps|_{E^s_T}+|\eps\partial_{x_d}U^n_\eps|_{L^\infty}\leq K \text{ for }\eps\in (0,\eps_1].
\end{align}
Then there exist  constants $T_0(K)$ and $\eps_0(K)\leq \eps_1$ such that for $0<\epsilon\leq\epsilon_0$ and  $T\leq T_0$ we have
\begin{align}\label{e10}
|U^{n+1}_\epsilon|_{E^s_T}+\sqrt{T}\langle U^{n+1}_\epsilon\rangle_{s+1,T}\leq C(K,E)\sqrt{T}\left(|\cF|_{E^s_T}+\langle G\rangle_{s+1,T}\right).
\end{align}
\end{prop}

The desired boundedness of $\cU^0_\eps$ and the $\cU^{0,n}_\eps$ in $E^s_{T_0}$ follows from the boundedness of $\cV^0$ and the uniform boundedness of the $\cV^{0,n}$ in $\bH^{s+1}_{T_0}$, considered with Proposition \ref{propA} and Lemma \ref{e11}. Similarly, from the fact that
\begin{align}
\lim_{n\rightarrow\infty} \cV^{0,n} = \cV^0 \textrm{ in } \bH^{s}_{T_0},
\end{align}
we get \eqref{e8}(b).

Now we proceed by proving \eqref{e8}(c). Fix $\delta>0$. Noting $\bE\cV^{0,n}=\cV^{0,n}$ and $\bE\cV^{0,n+1}=\cV^{0,n+1}$, choose finite partial trigonometric polynomial sums $\cV^{0,n}_p$ and $\cV^{0,n+1}_p$ of (resp.) $\cV^{0,n}$ and $\cV^{0,n+1}$ such that
\begin{align}\label{e17}
\begin{split}
&(a)\;\bE\cV^{0,n}_p=\cV^{0,n}_p \text{ and }\bE\cV^{0,n+1}_p=\cV^{0,n+1}_p,\\
&(b)\;|\cV^{0,n}-\cV^{0,n}_p|_{H^{s+1}_{T_0}}<\delta, |\cV^{0,n+1}-\cV^{0,n+1}_p|_{H^{s+1}_{T_0}}<\delta,
\end{split}
\end{align}
and note that, as a result of \eqref{e17}(b), we also have
\begin{align}
|\partial_{x_d}\cV^{0,n+1}-\partial_{x_d}\cV^{0,n+1}_p|_{H^s_{T_0}}<\delta.
\end{align}
Observe then, as a consequence of Proposition \ref{propA}, we get
\begin{align}\label{new36}
|\cU^{0,n}-\cU^{0,n}_p|_{\cE^s_{T_0}}<C\delta, |\cU^{0,n+1}-\cU^{0,n+1}_p|_{\cE^s_{T_0}}<C\delta,  \text{ and }|\partial_{x_d}\cU^{0,n+1}-\partial_{x_d}\cU^{0,n+1}_p|_{\cE^{s-1}_{T_0}}<C\delta,
\end{align}
where $\cU^{0,n}_p$ and $\cU^{0,n+1}_p$ are obtained from (resp.) $\cV^{0,n}_p$ and $\cV^{0,n+1}_p$ via the substitution $\theta=\theta(\theta_0,\xi_d)$. The induction assumption is
\begin{align}
\lim_{\epsilon\rightarrow 0}|U^n_\eps-\cU^{0,n}_\eps|_{E^{s-1}_{T_0}}=0.
\end{align}
With the boundedness of the $U^n_\eps$ in $E^s_{T_0}$, it follows
\begin{align}\label{new45}
\lim_{\eps\to 0}|F(\eps U^n_\eps)(U^n_\eps)-F(0)\cU^{0,n}_\eps |_{E^{s-1}_{T_0}}=0.
\end{align}
Observe that since $\cV^{0,n}$ and $\cV^{0,n}_p$ are invariant under $\bE$, Lemma \ref{e11} applies. Thus, using the estimate \eqref{e13}, from \eqref{new45}  and \eqref{new36} we obtain
\begin{align}\label{new37}
|F(\eps U^n_\eps)(U^n_\eps)-F(0)\cU^{0,n}_{p,\eps} |_{E^{s-1}_{T_0}}\leq C\delta+c(\eps),
\end{align}
where $c(\eps)\to0$ as $\eps\to0$.

 Now we define
\begin{align}\label{d2}
\mathcal{G}_p=\tilde{L}(\partial_x)\mathcal{V}^{0,n+1}_p+\mathcal{M}(\mathcal{V}^{0,n}_p)\partial_\theta\mathcal{V}^{0,n+1}_p.
\end{align}
Then
\begin{align}
\begin{split}
\mathbb{E}(\mathcal{G}_p-F(0)\mathcal{V}^{0,n}_p)
=& \mathbb{E}(\tilde{L}(\partial_x)(\cV^{0,n+1}_p-\cV^{0,n+1})) +\\
& \mathbb{E}(\cM(\cV^{0,n}_p)\partial_\theta \cV^{0,n+1}_p-\cM(\cV^{0,n})\partial_\theta \cV^{0,n+1}) +\\
& \mathbb{E}(F(0)(\cV^{0,n}-\cV^{0,n}_p)) +R^{n+1}.
\end{split}
\end{align}
Thus using \eqref{e17}(b) and continuity of both $\bE:H^{s;2}\rightarrow H^{s;1}$ and multiplication from $H^{s;1}\times H^{s;1}\rightarrow H^{s;2}$, we get
\begin{align}\label{d3}
|\mathbb{E}\mathcal{G}_p-\mathbb{E}(F(0)\mathcal{V}^{0,n}_p)-R^{n+1}|_{H^s_{T_0}}=O(\delta).
\end{align}

We define the operator
\begin{align}\label{e25}
\bL_0=\tilde L(\partial_x)+\frac{1}{\eps}\tilde L(d\phi_0)\partial_{\theta_0} + \cM'(\cU^{0,n}_{p,\eps})\partial_{\theta_0}.
\end{align}
We claim
\begin{align}\label{e26}
|\bL_0 U^{n+1}_\eps-F(\eps U^n_\eps)U^n_\eps|_{E^{s-1}_{T_0}}\leq C\delta+c(\eps),
\end{align}
where $c(\eps)\to 0$ as $\eps\to 0$. This follows from \eqref{e5}(a) and
\begin{align}\label{e27}
\begin{split}
&|\tilde A_j(\eps U^n_\eps)\partial_{x_j}U^{n+1}_\eps-\tilde A_j(0)\partial_{x_j}U^{n+1}_\eps|_{E^{s-1}_{T_0}}=O(\eps)\\
&\left|\frac{1}{\eps}\tilde A_j(\eps U^n_\eps)\beta_j\partial_{\theta_0}U^{n+1}_\eps-\left(\frac{1}{\eps}\tilde A_j(0)\beta_j\partial_{\theta_0}U^{n+1}_\eps+\partial_u\tilde A_j(0) U^n_\eps\beta_j\partial_{\theta_0}U^{n+1}_\eps\right)\right|_{E^{s-1}_{T_0}}=O(\eps)\\&\left|\partial_u\tilde A_j(0)(U^n_\eps-\cU^{0,n}_{p,\eps})\beta_j\partial_{\theta_0}U^{n+1}_\eps\right|_{E^{s-1}_{T_0}}\leq C|U^n_\eps-\cU^{0,n}_{p,\eps}|_{E^{s-1}_{T_0}}\leq c(\eps)+O(\delta).
\end{split}
\end{align}

Setting $\cG'_p=\tilde L(\partial_x)\cU^{0,n+1}_p +\cM'(\cU^{0,n}_p)\partial_{\theta_0}\cU^{0,n+1}_p$, since $\cL'(\partial_{\theta_0},\partial_{\xi_d})\cU^{0,n+1}_p=0$, we get
\begin{align}\label{e28}
\bL_0\cU^{0,n+1}_{p,\eps}=\cG'_{p,\eps}.
\end{align}
From now on we use $|_{(\theta_0,\xi_d)}$ to denote evaluation at $\theta=\theta(\theta_0,\xi_d)$, and $|_{(\theta_0,\xi_d),\epsilon}$ to indicate $|_{(\theta_0,\xi_d)}$ followed by evaluation at $\xi_d=\frac{x_d}{\epsilon}$. It is easy to check $\mathcal{G}'_p=\mathcal{G}_p|_{(\theta_0,\xi_d)}$, particularly because $\cV^{0,n}_p$, $\cV^{0,n+1}_p$  are finite polynomials. Thus
\begin{eqnarray}
\mathbb{L}_0\mathcal{U}^{0,n+1}_{p,\epsilon}-F(0)\mathcal{U}^{0,n}_{p,\epsilon}
&=& \mathcal{G}'_{p,\epsilon}-F(0)\mathcal{U}^{0,n}_{p,\epsilon} \nonumber \\
&=& [\mathcal{G}_p-F(0)\mathcal{V}^{0,n}_p]|_{(\theta_0,\xi_d),\epsilon} \\
&=& [(\mathbb{E}(\mathcal{G}_p-F(0)\mathcal{V}^{0,n}_p)-R^{n+1})+R^{n+1}+(I-\mathbb{E})(\mathcal{G}_p-F(0)\mathcal{V}^{0,n}_p)]|_{(\theta_0,\xi_d),\epsilon}\label{new19}
\end{eqnarray}
We claim:\\
(i)
\begin{align}\label{new20}
|(\mathbb{E}(\mathcal{G}_p-F(0)\mathcal{V}^{0,n}_p)-R^{n+1})|_{(\theta_0,\xi_d),\epsilon}|_{E^{s-1}_{T_0}}=O(\delta),
\end{align}
(ii)
\begin{align}\label{new4}
\lim_{\epsilon\rightarrow 0}\left| R^{n+1} |_{(\theta_0,\xi_d),\epsilon}\right|_{E^{s-1}_{T_0}(x,\theta_0)}=0,
\end{align}
and (iii) there exists $\mathcal{V}^1_p\in H^{s;2}_{T_0}$ such that for $\mathcal{U}^1_p=\mathcal{V}^1_p|_{\theta(\theta_0,\xi_d)}$,
\begin{align}\label{new17}
\mathcal{L}'(\partial_{\theta_0},\partial_{\xi_d})\mathcal{U}^1_p=-[(I-\bE)(\mathcal{G}_p-F(0)\mathcal{V}^{0,n}_p)]|_{\theta(\theta_0,\xi_d)}.
\end{align}
(i) follows from \eqref{d3} followed by application of Proposition \ref{propA} and then Lemma \ref{e11}.\\
To see (ii), observe it is clear from our definition of $R^{n+1}$ that the continuity of $\bE$ implies $R^{n+1}(x,\theta)\in H^{s;1}_{T_0}(x,\theta)$, and recall from Remark \ref{new5} that we have $\bE_e R^{n+1}= R^{n+1}$ and \eqref{new6}. By requiring \eqref{rename6}, and thus \eqref{b2}, to hold at $x_d=0$, we have obtained $R^{n+1}|_{x_d=0}=0$. So Proposition \ref{new3} grants us \eqref{new4}.\\
Regarding (iii), note that, by Remark \ref{new18}, \eqref{new17} holds if and only if
\begin{align}
\mathcal{L}'(\partial_{\theta_0},\partial_{\xi_d})\mathcal{U}^1_p=-[(I-\bE^\flat)(\mathcal{G}_p-F(0)\mathcal{V}^{0,n}_p)]|_{\theta(\theta_0,\xi_d)},
\end{align}
for which a sufficient condition is that, where $\mathcal{U}^1_p=\mathcal{V}^1_p|_{\theta(\theta_0,\xi_d)}$,
\begin{align}
\mathcal{L}(\partial_\theta)\mathcal{V}^1_p=-(I-\bE^\flat)(\mathcal{G}_p-F(0)\mathcal{V}^{0,n}_p).
\end{align}
Proposition \ref{new7} guarantees the existence of such $\mathcal{V}^1_p$, so proof of (iii) is complete.\\
Noting 
\begin{align}
\mathbb{L}_0(\epsilon \mathcal{U}^1_{p,\epsilon})=(\mathcal{L}'(\partial_{\theta_0},\partial_{\xi_d})\mathcal{U}^1_p)_\epsilon+(\tilde L(\partial_x)\epsilon \mathcal{U}^1_p)_\epsilon+\mathcal{M}'(\mathcal{U}^{0,n}_{p,\epsilon})\partial_{\theta_0}(\epsilon\mathcal{U}^1_{p,\epsilon}),
\end{align}
it follows from \eqref{new19}-\eqref{new17} that
\begin{align}\label{new38}
|\mathbb{L}_0(\mathcal{U}^{0,n+1}_{p,\epsilon}+\epsilon\mathcal{U}^1_{p,\epsilon})-F(0)\mathcal{U}^{0,n}_{p,\epsilon}|_{E^{s-1}_{T_0}}
\leq C\delta+c(\epsilon)+K(\delta)\epsilon,
\end{align}
where $c(\epsilon)$ has been altered, but still tends to zero as $\epsilon\rightarrow 0$. It follows from equations \eqref{new37}, \eqref{e26}, and \eqref{new38} that
\begin{align}\label{new38b}
|\mathbb{L}_0(U^{n+1}_\eps-(\mathcal{U}^{0,n+1}_{p,\epsilon}+\epsilon\mathcal{U}^1_{p,\epsilon}))|_{E^{s-1}_{T_0}}
\leq C\delta+c(\epsilon)+K(\delta)\epsilon.
\end{align}

Now we claim that we have the following estimates:
\begin{align}\label{e34}
&\begin{split}
&(a)\;\left|\left(\partial_{x_d}+\bA(\eps \cU^{0,n}_{p,\eps},\partial_{x'}+\frac{\beta\cdot\partial_{\theta_0}}{\eps})\right)\left(U^{n+1}_\eps-(\cU^{0,n+1}_{p,\eps}+\eps\cU^1_{p,\eps})\right)\right|_{E^{s-1}_{T_0}}\leq C\delta +c(\eps)+K(\delta)\eps\\
&(b)\left|B(\eps \cU^{0,n}_{p,\eps})\left(U^{n+1}_\eps-(\cU^{0,n+1}_{p,\eps}+\eps\cU^1_{p,\eps})\right)\right|_{H^s_{T_0}}\leq C\delta +c(\eps)+K(\delta)\eps.
\end{split}
\end{align}
That \eqref{e34}(a) holds follows from \eqref{new38b} with the use of estimates similar to \eqref{e27}, and \eqref{e34}(b) follows from the boundary conditions \eqref{new61}(b) and \eqref{fullConditionsIt}(c) together with Proposition \ref{propA} and Lemma \ref{e11}. After applying the estimate of Proposition \ref{e9}, we get that
\begin{align}
|U^{n+1}_\eps-(\mathcal{U}^{0,n+1}_{p,\epsilon}+\epsilon\mathcal{U}^1_{p,\epsilon})|_{E^{s-1}_{T_0}}
\leq C\delta+c(\epsilon)+K(\delta)\epsilon.
\end{align}
from which we conclude
\begin{align}
|U^{n+1}_\epsilon-\mathcal{U}^{0,n+1}_\epsilon|_{E^{s-1}_{T_0}}\leq C\delta + c(\epsilon)+K(\delta)\epsilon.
\end{align}
This finishes the induction step, completing the proof of the theorem.
\end{proof}

\end{document}